\documentclass{article}

\usepackage[utf8]{inputenc}
\usepackage[T1]{fontenc}
\usepackage{amsmath, amssymb, amsthm}
\usepackage{geometry}
\usepackage{hyperref}
\usepackage{graphicx}
\usepackage{enumerate}
\usepackage{xcolor}
\usepackage{bm}
\usepackage{float}
\usepackage{booktabs}

\newcommand{\Ltwo}{L^2\left(\mathbb{S}^2 \right)} 
\newcommand{\sphere}{\mathbb{S}^{2}}
\newcommand{\Y}{Y_{\ell,m}}
\newcommand{\Yconj}{\overline{Y}_{\ell,m}}
\newcommand{\suml}{\sum_{\ell \geq 1}}
\newcommand{\summ}{\sum_{m=-\ell }^{\ell }}

\newcommand{\alm}{a_{\ell,m}}
\newcommand{\almG}{a_{\ell,m;G}}
\providecommand{\Ex}[1]{\mathbb{E}\left[#1\right]}
\newcommand{\diff}{\ensuremath{\,\mathrm{d}}}
\providecommand{\el}[1]{\ell_{#1}}
\newcommand{\mm}[1]{m_{#1}}
\providecommand{\Ybis}[2]{Y_{\ell_{#1},m_{#2}}}
\providecommand{\almbis}[2]{a_{\ell_{#1},m_{#2}}}
\providecommand{\walmbis}[2]{\widetilde{a}_{\ell_{#1},m_{#2}}}
\newcommand{\almtwo}{a_{\ell,m;2}}
\newcommand{\cumq}[1]{\mathrm{Cum}_4\left(#1\right)}

\newcommand{\wigner}[6]{ \begin{pmatrix}
	#1 & #2 & #3 \\ 
	#4 & #5 & #6%
\end{pmatrix}}

\newcommand{\wignersix}[6]{ \begin{Bmatrix}
		#1 & #2 & #3 \\ 
		#4 & #5 & #6%
\end{Bmatrix}}

\newcommand{\Var}[1]{\mathrm{Var}\left(#1\right)}
\newcommand{\Cov}[2]{\mathrm{Cov}\left(#1,#2\right)}

\newcommand{\fNL}{f_{\mathrm{NL}}}

\newcommand{\ClG}{C_{\ell;G}}
\newcommand{\Clbis}[1]{C_{\ell_{#1}}}
\newcommand{\lcard}{2\ell+1}
\newcommand{\lcardbis}[1]{2\ell_{#1}+1}

\newcommand{\walm}{\widetilde{a}_{\ell,m}}
\newcommand{\wT}{\widetilde{T}}
\newcommand{\cg}{C^{\ell,m}_{\el{1},\mm{1};\el{2},\mm{2}}}
\newcommand{\cgnull}{C^{\ell,0}_{\el{1},0;\el{2},0}}
\newcommand{\cgbis}[6]{C^{#1,#2}_{#3,#4;#5,#6}}
\newcommand{\gaunt}{\mathcal{G}\left(\el{1},\mm{1};\el{2},\mm{2};\ell,m\right)}
\newcommand{\gauntbis}{\mathcal{G}\left(\el{1},\mm{1};\el{2},\mm{2};\ell,-m\right)}

\newcommand{\Bisp}{B_{\ell _{1},\ell _{2},\ell _{3}}}
\newcommand{\bisp}{b_{\ell _{1},\ell _{2},\ell _{3}}}
\newcommand{\Gammal}{\Gamma_{\el{1},\el{2},\el{3}}}
\newcommand{\hal}{h_{\el{1},\el{2},\el{3}}}
\newcommand{\etal}{\eta_{\el{1},\el{2},\el{3}}} 
\newcommand{\Bispest}{\widehat{B}_{\ell _{1},\ell _{2},\ell _{3}}}

\newcommand{\fNLest}{\widehat{f}_{\text{NL}}}

\newcommand{\TG}{T_{G}}
\newcommand{\Plm}{P_{\ell, m}}

\providecommand{\cc}[1]{\overline{#1}}

\newcommand{\argmin}{\operatorname{argmin}}

\newcommand{\dW}{d_{\operatorname{TV}}}
\newcommand{\CW}{C_{TV;\alpha,r}}
\providecommand{\abs}[1]{\left\vert #1 \right \vert}
\newcommand{\LambdaL}{	\Lambda_{L,L_0} }
\providecommand{\summinus}{\sum_{\left(\ell_1,\ell_2,\ell_3\right) \in \LambdaL}}

\newcommand{\Seta}{S_ {\eta^2;L,L_0}}
\newcommand{\Skappa}{S_{\kappa;L,L_0}}
\newcommand{\Ar}{A_r}
\newcommand{\kappal}{\kappa_{\ell_1,\ell_2,\ell_3}}

\newtheorem{theorem}{Theorem}
\newtheorem{condition}[theorem]{Condition}
\newtheorem{corollary}[theorem]{Corollary}

\newtheorem{lemma}[theorem]{Lemma}

\newtheorem{proposition}[theorem]{Proposition}
\newtheorem{remark}[theorem]{Remark}

\title{Gaussian approximation for non-linearity parameter estimation in perturbed random fields on the sphere}
\author{Claudio Durastanti\thanks{Department of Basic Sciences and Applied for Engineering, Sapienza University of Rome, \texttt{claudio.durastanti@uniroma1.it}} 
}
\date{\today}  

\begin{document}
	
	\maketitle

\begin{abstract}
We develop a probabilistic framework for the asymptotic analysis of a bispectrum-based estimator of primordial non-Gaussianity for isotropic random fields on the sphere in the high-resolution regime. By reformulating the estimation problem as an ordinary least squares regression, we derive the asymptotic moments of the estimator. Combining these results with Stein–Malliavin techniques on Wiener chaos yields a quantitative Gaussian approximation with an explicit convergence rate in total variation distance. The analysis relies on sharp asymptotic estimates for the deterministic weights arising from spherical harmonic coupling coefficients. Numerical experiments illustrate the predicted scaling laws and provide qualitative evidence for the asymptotic Gaussian behavior.\\


\end{abstract}

\section{Introduction}\label{sec:intro}

The statistical analysis of non-linear functionals of random fields on manifolds constitutes a central theme in modern applied probability, with deep implications for spatial statistics and cosmological data analysis. A prime example is the estimation of the primordial non-Gaussianity parameter, denoted by $\fNL$, in local inflationary models. From a probabilistic perspective, this task can be formulated as the inference of a non-Gaussianity parameter from a single realization of an isotropic random field on the sphere $\mathbb{S}^2$, which requires characterizing the asymptotic distribution of the resulting estimator. While standard Gaussian random fields on the sphere are well-understood, non-Gaussian perturbations naturally generate complex higher-order dependencies (see, e.g., \cite{baas,planck2020b}), demanding robust quantitative central limit theorems (CLTs) for non-linear spherical functionals.

In this context, the Cosmic Microwave Background (CMB) temperature anisotropies are mathematically modeled as a realization of an isotropic spherical random field, whose departures from Gaussianity are encoded through its third-order dependence structure. Consequently, the angular bispectrum, the spherical analogue of the third-order polyspectrum, provides the natural statistic for detecting such non-Gaussian signatures. Under the local model of primordial non-Gaussianity, the reduced bispectrum depends linearly on $\fNL$. This algebraic structure motivated the introduction of the Komatsu--Spergel--Wandelt (KSW) estimator \cite{ksw05}, which has become a benchmark statistic in empirical cosmology (see, e.g., \cite{Mun14,planck2020a}). However, while the KSW estimator is widely used in practice, a rigorous mathematical characterization of its asymptotic distribution under high-frequency regimes has remained elusive due to the intricate combinatorial coupling of its weights.

The literature on spherical random fields has recently seen significant developments along two parallel lines. On one hand, physical and cosmological studies have extensively focused on implementing computationally efficient estimators for $\fNL$ (see, e.g., \cite{bartolo,cagliari24,komatsu10,yadavwandelt10}). On the other hand, a substantial body of probabilistic research has established central limit theorems, quantitative normal approximations, and functional limit theorems for spherical bispectrum statistics under Gaussianity (see, e.g., \cite{M2008,marpec4,MaPeCUP}). 
Crucially, these available probabilistic results focus on individual multipoles or smoothly normalized bispectrum functionals, and they do not directly apply to the KSW estimator. The KSW statistic is a heavily weighted linear functional over the entire set of admissible triplet configurations, where the weights are governed by Clebsch-Gordan coefficients and Wigner $3j$-symbols. Consequently, its asymptotic behavior cannot be deduced from standard ergodic arguments or existing limit theorems for independent blocks, necessitating a novel and dedicated probabilistic analysis.

In this paper, we bridge this gap by investigating the asymptotic behavior of a narrow-band version of the KSW estimator in the high-frequency regime, where the maximum multipole resolution $L$ tends to infinity. Specifically, we restrict attention to multipoles satisfying
\(
rL\le \ell\le L,
\)
for a fixed $r\in(0,1/2)$, thereby defining a narrow-band estimator while preserving the asymptotic $L^3$ scaling of the number of admissible multipole configurations.

Our main contribution is a quantitative central limit theorem for the KSW-type narrow-band estimator $\fNLest$ in this high-frequency regime. We prove that
\[
d_{\mathrm{TV}}
\!\left(
\frac{\widehat f_{\mathrm{NL}}-f_{\mathrm{NL}}}
{\sqrt{\operatorname{Var}(\widehat f_{\mathrm{NL}})}},
Z
\right)
\le C L^{-2},
\]
where $Z \sim \mathcal{N}(0,1)$ and $C>0$ is a constant independent of $L$. This establishes an explicit, quantitative rate of convergence in the total variation distance. The proof relies on a combination of asymptotic estimates for Wigner $3j$-symbols, properties of Gaunt integrals, and a bound on fourth-order cumulants via Stein--Malliavin calculus on the Wiener chaos (see, e.g., \cite{noupebook}). This machinery allows us to control the complex spatial dependence induced across all scales by the triangular configurations. To the best of our knowledge, this provides the first quantitative CLT for the KSW estimator, introducing a methodological framework that can be extended to other weighted statistics on the sphere.

The paper is organized as follows. Section~\ref{sec:preliminary} reviews the background on isotropic spherical random fields and bispectrum statistics. Section~\ref{sec:OLS} introduces the KSW estimator, states the main quantitative central limit theorem, and develops the auxiliary results needed for its proof. Section~\ref{sec:numerics} presents some numerical experiments. Section~\ref{sec:proofs} is devoted to the proofs.

\section{Preliminary results}\label{sec:preliminary}
In this section, we collect the necessary probabilistic background on isotropic spherical random fields, their chaotic decompositions, and the associated harmonic analysis tools. We then formalize the non-Gaussian perturbed model and its third-order spectral characteristics, which underpin the construction of the weighted statistics analyzed in the sequel. Standard references include \cite{MaPeCUP,steinweiss,yadrenko}.

Let $\sphere$ denote the unit sphere in $\mathbb{R}^3$ equipped with its standard Lebesgue uniform measure $\diff x = \sin\vartheta \diff\vartheta \diff\phi$, where $x =\left(\vartheta,\varphi\right)$ represents the standard spherical coordinates. We denote by $\Ltwo=L^2\left(\mathbb{S}^2,\diff x \right)$ the Hilbert space of square-integrable complex-valued functions on the sphere.

\subsection{Gaussian random fields and their stochastic properties}
Let $\{Y_{\ell,m}:\ell \geq 1, m=-\ell ,\dots,\ell\}$ denote the standard complex spherical harmonics, which form an orthonormal basis of $L^2(\mathbb S^2)$ (see \cite{MaPeCUP,steinweiss,vilenkin}). Following for example \cite{szego}, the representation is given as follows
\begin{equation*}
\Y \left(x\right)= 	\Y \left(\vartheta,\varphi\right)=\sqrt{\frac{\lcard}{4\pi} \frac{\left(\ell-m\right)!}{\left(\ell+m\right)!}}\Plm \left(\cos \vartheta \right) e^{i m \varphi},
\end{equation*}
where $\Plm$ is the associated Legendre polynomial of degree $\ell$ and order $m$.

Let $\TG = \{\TG(x): x \in \sphere\}$ denote a centered, mean-square continuous, Gaussian random field defined on a probability space $(\Omega, \mathcal{F}, \mathbb{P})$. The field $\TG$ is strictly isotropic if its law is invariant under the action of the special orthogonal group $SO(3)$, namely,
\begin{align*}
	\TG\left(R x\right) \overset{d}{=} \TG\left(x\right) \quad \text{ for all } R \in SO(3), \quad x \in \sphere,
\end{align*}
where $\overset{d}{=}$ denotes equality in distribution. 
By the Peter--Weyl theorem, $\TG$ admits the spectral Karhunen--Lo\`eve expansion, converging in the $L^2(\Omega \times \mathbb{S}^2)$ sense:
\begin{equation}\label{eq:TG}
\TG\left( x\right) =\suml \summ \almG \Y \left( x\right) , \quad x\in \sphere,
\end{equation}%
where the random and centered harmonic coefficients $\left\lbrace\almG:\ell \geq 1, m=-\ell ,\dots,\ell\right\rbrace$ are given by 
\begin{equation*}
\almG=\langle \Y, \TG\rangle_{\sphere}=\int_{\mathbb{S}^{2}}\TG\left( x\right) \Yconj \left( x\right) \diff x.
\end{equation*}
Here, $\cc{z}$ denotes the complex conjugate of $z \in \mathbb{C}$. 

A direct consequence of isotropy is that the covariance operator is diagonalized in the harmonic basis. Specifically, the coefficients satisfy
\begin{equation}
\Ex{ \almG \cc{a}_{\ell^\prime,m^\prime;G} }= \ClG \delta_{\ell}^{\ell^\prime} \delta_{m}^{m^\prime}, \label{eq:powerdef}
\end{equation}
where $\delta_{\cdot}^{\cdot}$ is the Kronecker delta. The sequence $\left\{\ClG:\ell \geq 1 \right\}$ forms the angular power spectrum of the field. Correspondingly, by the addition theorem for spherical harmonics, the covariance function depends solely on the geodesic distance $\langle x, y \rangle$ between points $x,y \in \mathbb{S}^2$:
\begin{equation*}
\mathbb{E}[\TG(x)\overline{\TG}(y)] = \suml \frac{2\ell + 1}{4\pi} \ClG P_\ell\left(\langle x, y \rangle \right),
\end{equation*}
where $P_\ell: [-1,1] \to \mathbb{R}$ is the standard Legendre polynomial of degree $\ell$. To ensure regular high-frequency behavior of the random field, we impose the following spectral decay condition, which is standard in the spatial asymptotics literature (see, e.g., \cite{dlm,marpec2}).
\begin{condition}\label{cond:powerspectrum}
\label{powerspectrum}Let $\ClG$ be defined as in \eqref{eq:powerdef}. For any $\ell \geq 0$, there exist an amplitude $A>0$, and a spectral index $\alpha >2$ so that%
\begin{equation*}
\ClG=A\left(1+\ell\right) ^{-\alpha }.  \label{eq:powercond}
\end{equation*}
\end{condition}
Condition \ref{cond:powerspectrum} is standard in the literature on isotropic spherical random fields and is satisfied by several models arising in cosmology, including the Sachs--Wolfe model (see, for example, \cite{dode2004}). 

Under Gaussianity and isotropy, the set of coefficients $\lbrace \almG: \ell \geq 1, m=-\ell,\ldots,\ell \rbrace$ forms a collection of independent centered complex Gaussian random variables, with covariance given by \eqref{eq:powerdef} (see, e.g., \cite{ls15,MaPeCUP}).
Consequently, all higher-order moments are determined by Wick's theorem: odd moments vanish, whereas every even-order moment is obtained by summing over all pairings. For every $k\geq1$,
\[
\mathbb E\!\left[\prod_{j=1}^{2k}a_{\ell_j,m_j;G}\right]
=
\sum_{\pi\in\mathcal P_{2k}}
\prod_{(i,j)\in\pi}
\mathbb E\!\left[
a_{\ell_i,m_i;G}\,
\overline{a}_{\ell_j,m_j;G}
\right],
\]
where the sum runs over all pairings of $\{1,\ldots,2k\}$ (see, e.g., \cite[Chapter~5]{MaPeCUP}). This identity will be repeatedly used throughout the paper to evaluate higher-order moments and cumulants.

\subsection{Algebraic structure of angular coupling: Wigner and Gaunt symbols}
\label{sec:wigner}  

Non-linear transformations of spherical random fields naturally involve products of basis functions, whose harmonic projections are governed by angular momentum coupling coefficients. We recall the core algebraic properties of Wigner $3j$-symbols and Clebsch--Gordan coefficients required for tracking spatial dependencies (see \cite{MaPeCUP,vmk} for details).

The Wigner $3j$-symbols, denoted by $\wigner{\el{1}}{\el{2}}{\el{3}}{m_1}{m_2}{m_3}$, are real coefficients satisfying the following properties \cite[Proposition 3.44]{MaPeCUP}:
  \begin{enumerate}
  	\item the Wigner $3j$-coefficients are nonzero only if the following triangle conditions hold
  	\begin{align}\label{eq:ell}
  		&	\abs{\el{i}-\el{j}} \leq
  		\el{k} \leq \el{i} +\el{j}, \quad i\neq j \neq k = 1, 2, 3\\
  		& m_1 + m_2 + m_3 = 0;\label{eq:m}
  	\end{align}
  	\item the Wigner $3j$-coefficients are invariant under
  	permutations of any two columns when the sum $\el{1}+\el{2}+\el{3}$ is even; 
  	\item the following upper bound holds for any $\el{1}, \el{2}, \el{3}$:
  	\begin{align}\label{eq:wignerup}
  		\wigner{\el{1}}{\el{2}}{\el{3}}{m_1}{m_2}{m_3}\leq 
  		\left( \max\left[\left(2\el{1}+1\right)\left(2\el{2}+1\right)\left(2\el{3}+1\right)\right] \right)^{-\frac{1}{2}}.
  	\end{align}
  	\item If the triangle conditions \eqref{eq:ell}--\eqref{eq:m} hold and $m_1 = m_2 = m_3 = 0$, the
  	Wigner $3j$-coefficients are nonzero only if the sum $\el{1} + \el{2} + \el{3}$ is even;
  	\item For every choice of $\el{1}, \el{2}, \el{3}$, the following sign inversion rule holds 
  	\begin{align*}	
  		\wigner{\el{1}}{\el{2}}{\el{3}}{m_1}{m_2}{m_3}
  		=    \left( -1\right)^{\el{1}+\el{2}+\el{3}} 	\wigner{\el{1}}{\el{2}}{\el{3}}{-{m_1}}{-{m_2}}{-{m_3}}.
  	\end{align*}
  \end{enumerate}

  The Clebsch--Gordan coefficients $\cg$ are closely related to the Wigner $3j$-symbols through the relations
  \begin{equation}
	\begin{split}
  	& \wigner{\el{1}}{\el{2}}{\ell}{\mm{1}}{\mm{2}}{m} = \left(-1\right)^{\ell+m} \frac{1}{\sqrt{\lcard}} \cgbis{\ell}{m}{\el{1}}{-\mm{1}}{\el{2}}{-\mm{2}}, \\ 
  	&\cg = \left(-1\right)^{\el{1}-\el{2}+m} \sqrt{\lcard} \wigner{\el{1}}{\el{2}}{\ell}{\mm{1}}{\mm{2}}{-m}.\label{eq:wignercj},
  \end{split}
  \end{equation}
see, for further details, \cite{MaPeCUP,vmk}. They satisfy the following orthogonality and summation relations 
  \begin{align}
  	&\cgbis{\ell^\prime}{m^\prime}{\ell_2}{m_2}{\ell_1}{m_1}=(-1)^{\ell_1+\ell_2-\ell}\cgbis{\ell^\prime}{m^\prime}{\ell_1}{m_1}{\ell_2}{m_2}\label{eq:prop0}\\
  	& \sum_{m_1}(-1)^{\ell_1-m_1}\cgbis{\ell}{0}{\ell_1}{m_1}{\ell_1}{-m_1} = \sqrt{2\ell_1+1} \delta_{0}^{\ell} \label{eq:prop1}\\
  	&\cgbis{\ell}{0}{\ell_1}{m_1}{\ell_2}{-m_2} = \frac{(-1)^{\ell_1-m_1} \delta_{\ell_2}^{\ell_1}\delta_{m_2}^{m_1}}{\sqrt{2\ell_1+1} } \label{eq:prop2},\\
  	&\sum_{m_1,m_2} \cg\cgbis{\ell^\prime}{m^\prime}{\ell_1}{m_1}{\ell_2}{m_2}=\delta_{\ell}^{\ell^\prime}\delta_{m}^{m^\prime} \label{eq:prop3},
  \end{align}
  see \cite[Chapter 8]{vmk}.

The main object linking these coefficients with nonlinear functionals of spherical random fields is the Gaunt integral, namely the integral of the product of three spherical harmonics. The Gaunt integral $\gaunt$ admits the representations
  \begin{equation}\begin{split}
  & \gaunt  = \int_{\sphere} \Ybis{1}{1}\left( x \right)\Ybis{2}{2}\left( x \right)\Yconj \left( x \right) \diff x\\
  	& \quad\quad\quad =\left(-1\right)^{m} \sqrt{\frac{\left(\lcardbis{1}\right)\left(\lcardbis{2}\right)\left(\lcard\right)}{4\pi}} \wigner{\el{1}}{\el{2}}{\ell}{\mm{1}}{\mm{2}}{-m} \wigner{\el{1}}{\el{2}}{\ell}{0}{0}{0} \\
  	& \quad \quad\quad= \sqrt{\frac{\left(\lcardbis{1}\right)\left(\lcardbis{2}\right)}{4\pi\left(\lcard\right)}} \cgbis{\ell}{m}{\el{1}}{\mm{1}}{\el{2}}{\mm{2}}\cgnull. \label{eq:gaunt}
  \end{split}\end{equation}
  The Gaunt integral provides the harmonic decomposition of products of spherical harmonics and will repeatedly appear in the analysis of the perturbed field introduced in the next subsection.

\subsection{Non-linear perturbed random fields}
We consider a non-Gaussian random field $T = \{T(x): x \in \sphere\}$ constructed as a controlled quadratic perturbation of the reference Gaussian field $T_G$. Formally, we define
\[
T(x)
=
T_G(x)
+
f_{\mathrm{NL}}
H_2(T_G(x)),
\qquad x\in\mathbb S^2,
\]
where $H_2(u)=u^2-1$ is the second-order Hermite polynomial and $f_{\mathrm{NL}}$ scales the intensity of the non-Gaussian component. In the terminology of Malliavin calculus, $H_2(T_G(x))$ represents a point-wise mapping into the second Wiener chaos.

The harmonic expansion of $T$ is given by
 \begin{equation}\label{eq:Texpansion}
 T\left(x\right)=\suml \summ \alm \Y \left(x\right) = \suml \summ \left(\almG \left(x\right) + \fNL \almtwo\right) \Y \left(x\right),
 \end{equation}
where 
\begin{equation}\label{eq:Tharmonic}
 \alm=\int_{\sphere}T\left( x\right) \Yconj \left( x\right) \diff x= \almG + \fNL  \almtwo.
\end{equation}%
and the harmonic coefficients of $H_2(T_G)$ are given by
\begin{equation*}\begin{split}
\almtwo 
&= \int_{\sphere} H_2\left( T_{G}\left( x\right)
\right) \Yconj  \left( x\right) \diff x \\
& = \sum_{\ell _{1}\geq 1} \sum_{\ell _{2}\geq 1} \sum_{m_{1}=-\ell_{1}}^{\ell_{1}}  \sum_{m_{2}=-\ell_{2}}^{\ell_{2}}   \almbis{1}{1} \almbis{2}{2} \left(-1\right)^m \gauntbis ,
\end{split}\end{equation*}
Using the representation of the Gaunt integral \eqref{eq:gaunt}, it follows that 
\begin{equation}
\begin{split} \almtwo& = \sum_{\ell _{1},\ell _{2}\geq 0}  \sum_{m_{1},m_{2}}   a_{\ell_1,m_1;G} a_{\ell_2,m_2;G}  \sqrt{\frac{\left(\lcardbis{1}\right)\left(\lcardbis{2}\right)}{4\pi \left(\lcard\right)}}   \cg \cgnull\\
& =\sum_{\ell _{1},\ell _{2}\geq 0}  \sum_{m_{1},m_{2}} \left(-1\right)^m a_{\ell_1,m_1;G} a_{\ell_2,m_2;G}  \sqrt{\frac{\left(\lcardbis{1}\right)\left(\lcardbis{2}\right)\left(\lcard\right)}{4\pi}} \\
\label{eq:almtwo} &\times \wigner{\el{1}}{\el{2}}{\ell}{m_{1}}{m_{2}}{-m} \wigner{\el{1}}{\el{2}}{\ell}{0}{0}{0}.
\end{split}
\end{equation}
The harmonic expansion of the quadratic perturbation $H_2(T_G)$ is therefore given by
\[
H_2(T_G)
=
\sum_{\ell,m}a_{\ell,m;2}Y_{\ell,m}.
\]

The second-order properties of this chaotic term are encapsulated in the following proposition, originally presented in \cite{m}, which follows from standard contractive properties of Wiener chaos projections (see also \cite{MaPeCUP} and Section~\ref{sec:proofs}). 
\begin{proposition}\label{prop:almtwo}
	For any $\ell \geq 1$, and $m=-\ell,\ldots,\ell$, let $a_{\ell,m;2}$ be defined by Equation \eqref{eq:almtwo}.
	Then $\almtwo$ is centered, with covariance structure given by 
	\begin{equation*}
		\Cov{a_{\ell,m;2}}{a_{\ell^\prime,m^\prime;2}} =C_{\ell;2} \delta^{\ell}_{\ell^{\prime}}\delta^{m}_{m^{\prime}},
	\end{equation*}
	where 
	\begin{equation*}
		\begin{split}
		C_{\ell;2} & =  2\sum_{\ell _{1},\ell _{2}}  \frac{\left(\lcardbis{1}\right)\left(\lcardbis{2}\right)}{4\pi\left(\lcard\right)} \left(\cgnull\right)^2C_{\ell_1;G}C_{\ell_2;G}\\
		&= 2\sum_{\ell _{1}\geq 1} \sum_{\ell _{2}\geq 1}  \Clbis{1}\Clbis{2} \frac{\left(\lcardbis{1}\right)\left(\lcardbis{2}\right)\left(\lcard\right)}{4\pi} \wigner{\el{1}}{\el{2}}{\ell}{0}{0}{0}^2.
		\end{split}
		\end{equation*}
Furthermore, the second-chaos field $H_2(T_G)$ is strictly isotropic and its harmonic coefficients satisfy
		\begin{align*}
		&	\Ex{H_2\left(T_G(x)\right)} = \sum_{\ell \geq 1} \frac{2\ell+1}{4\pi}C_{\ell;G};\\
		&	\Cov{H_2\left(T_G(x)\right)}{H_2\left(T_G(y)\right)} =\sum_{\ell \geq 1} \frac{2\ell+1}{4\pi}C_{\ell;2} P_{\ell}\left(\langle x,y \rangle\right),
		\end{align*}
for any $x,y \in \mathbb{S}^2$.
	\end{proposition}
The proof is a shortened version of the analogous result in \cite{m,MaPeCUP} and is reported in Section \ref{sec:proofs}.

The previous proposition yields the corresponding second-order properties of the perturbed field, drawn from \cite{m,MaPeCUP}.
\begin{proposition}\label{prop:complete}
	Let $\alm$ and $T$ be defined as in Equations \eqref{eq:Tharmonic} and \eqref{eq:Texpansion}. For any $x \in \mathbb{S}^2$, the field $T(x)$ is centered and isotropic. Its harmonic coefficients satisfy
\begin{equation*}
		\Ex{\alm}=0 \quad \text{and} \quad \Ex{\left\vert\alm\right\vert^2}= C_{\ell},
	\end{equation*}
	where 
	\[
	C_{\ell} = C_{\ell;G} + \fNL^2 C_{\ell;2}.
	\]
Consequently, for any $x,y \in \mathbb{S}^2$, it holds that
	\begin{equation*}
		\Cov{T(x)}{T(y)} = \sum_{\ell \geq 1} \frac{2\ell+1}{4\pi} C_{\ell} P_{\ell}\left(\langle x,y\rangle\right). 
	\end{equation*}

\end{proposition} 

For the remainder of our asymptotic analysis, it is mathematically convenient to define the standard normalized coefficients:
\begin{equation}\label{eq:walm}
\walm =\frac{\alm}{\sqrt{\ClG}}=\frac{\almG+\fNL \almtwo}{\sqrt{\ClG}},
\end{equation}%
and the corresponding field $\wT$ defined as
\begin{equation}\label{eq:wT}
	\wT(x) = \sum_{\ell \geq 0} \sum_{m=-\ell}^{\ell} \walm \Y(x), \qquad x \in \mathbb{S}^2.
\end{equation}

Then, the following corollary holds.
\begin{corollary}
	For any $\ell \geq 1$, $m=-\ell,\ldots,\ell$, let $\walm$ be defined as in Equation \eqref{eq:walm} 
\begin{align*}
&	\Ex{\walm}  =0; \\ 
&	\Ex{\left\vert\walm\right\vert^2}= 1 +  \fNL^2 \frac{C_{\ell;2}}{\ClG}.
\end{align*}
Also, for any $x,y \in \mathbb{S}^2$, it holds that
\begin{align*}
&\Ex{\wT(x)}= 0;\\
&\Cov{\wT(x)}{\wT(y)} = \sum_{\ell \geq 1} \frac{2\ell+1}{4\pi} \left( 1 + \fNL^2\frac{C_{\ell;2}}{C_{\ell;G}}\right) P_{\ell}\left(\langle x,y\rangle\right). 
\end{align*}
\end{corollary}
The proof follows directly the one of Proposition \ref{prop:complete} and it is here omitted for the sake of brevity.

\subsection{The angular bispectrum on the Third Wiener Chaos}

The angular bispectrum represents the third-order polyspectrum of a spherical random field and acts as the fundamental diagnostic tool for non-Gaussian extensions. For the normalized field $\widetilde T$, we define the full bispectrum by
\[
B_{\ell_1,\ell_2,\ell_3}^{m_1,m_2,m_3}
=
\mathbb E
\left[
\widetilde a_{\ell_1,m_1}
\widetilde a_{\ell_2,m_2}
\widetilde a_{\ell_3,m_3}
\right].
\]
Under Gaussianity,
\[
B^{m_1,m_2,m_3}_{\ell_1,\ell_2,\ell_3}=0,
\]
since odd moments of Gaussian random variables vanish.

Under strict isotropy, the full bispectrum inherits the selection rules of the Wigner symbols and can be compressed into a rotationally invariant bispectrum $B_{\ell_1,\ell_2,\ell_3}$ by projecting it onto the Wigner manifold:
\begin{equation}\label{eq:rotbis}
		B_{\el{1},\el{2},\el{3}} =\sum_{m_1,m_2,m_3} \wigner{\el{1}}{\el{2}}{\el{3}}{\mm{1}}{\mm{2}}{\mm{3}} \Ex{\walmbis{1}{1} \walmbis{2}{2}\walmbis{3}{3}},
\end{equation} 
which is independent of the azimuthal indexes $m_1, m_2$ and $m_3$.
The reduced bispectrum $b_{\ell_1,\ell_2,\ell_3}$ is defined through
\[
B_{\ell_1,\ell_2,\ell_3}
=
\begin{pmatrix}
\ell_1&\ell_2&\ell_3\\
0&0&0
\end{pmatrix}
b_{\ell_1,\ell_2,\ell_3},
\]
thereby removing the purely geometric contribution associated with the Wigner $3j$-symbol (see \cite{hu}).
Equivalently, the expectation of the harmonic coefficients admits the representation
\begin{equation}\label{eq:expwalm}
	\begin{split}
\mathbb{E}\left[ \widetilde{a}_{\ell _{1},m_{1}}\widetilde{a}_{\ell _{2},m_{2}}%
\widetilde{a}_{\ell _{3},m_{3}}\right] & =%
\begin{pmatrix}
\ell _{1} & \ell _{2} & \ell _{3} \\ 
m_{1} & m_{2} & m_{3}%
\end{pmatrix}%
\Bisp \\
& =
\begin{pmatrix}
\ell _{1} & \ell _{2} & \ell _{3} \\ 
m_{1} & m_{2} & m_{3}%
\end{pmatrix}%
\begin{pmatrix}
\ell _{1} & \ell _{2} & \ell _{3} \\ 
0 & 0 & 0%
\end{pmatrix}%
\bisp,
\end{split}
\end{equation}%
see \cite{m,marpec4} and the references therein.\\

To formulate the first-order behavior under our non-Gaussian framework, we introduce the auxiliary geometric factors:
\begin{equation}\label{hal}
h_{\ell_1,\ell_2,\ell_3} = \sqrt{ \frac{(2\ell_1+1)(2\ell_2+1)(2\ell_3+1)}{(4\pi)^3}},
\end{equation}
and the weighted covariance factor combinations $\Gamma_{\ell_1,\ell_2,\ell_3} = \gamma_{\ell_1,\ell_2} + \gamma_{\ell_1,\ell_3} + \gamma_{\ell_2,\ell_3}$, where $\gamma_{\ell_i,\ell_j} = 2C_{\ell_i}C_{\ell_j} / \sqrt{C_{\ell_1}C_{\ell_2}C_{\ell_3}}$. Combining these definitions, we introduce the deterministic kernel weight:
\begin{equation}\label{eq:eta}
\eta_{\ell_1,\ell_2,\ell_3} = 2\zeta_{\ell_1,\ell_2,\ell_3} \begin{pmatrix} \ell_1&\ell_2&\ell_3\\ 0&0&0 \end{pmatrix} h_{\ell_1,\ell_2,\ell_3} \Gamma_{\ell_1,\ell_2,\ell_3},
\end{equation}
where $\zeta_{\ell_1,\ell_2,\ell_3} = 1+\delta_{\ell_1}^{\ell_2} +\delta_{\ell_2}^{\ell_3} +3\delta_{\ell_1}^{\ell_3}$. 

The leading-order behavior of the normalized bispectrum under the local non-Gaussian model is summarized in the following proposition (see \cite{m}). The quantity $\eta_{\ell_1,\ell_2,\ell_3}$ will play the role of the deterministic regression coefficient throughout the remainder of the paper.
\begin{proposition}\label{prop:expect}
Let $\walm$ be defined as in Equation \eqref{eq:walm}. Then, for every admissible multipole configuration, it holds that 
	\begin{equation*}
		\mathbb{E}\left[ \widetilde{a}_{\ell _{1},m_{1}}\widetilde{a}_{\ell_{2},m_{2}}\widetilde{a}_{\ell _{3},m_{3}}\right]	= \fNL \etal + o_{\ell}(\fNL). 
	\end{equation*}
\end{proposition}

Combining Proposition~\ref{prop:expect} with the definition of the rotationally invariant bispectrum \eqref{eq:rotbis} gives
\begin{align}\label{eq:Bisp}
 \Bisp = &  \fNL \etal + o\left(\fNL\right) 
\end{align}
Replacing expectations in \eqref{eq:rotbis} by empirical averages leads to the sample bispectrum
\begin{equation}\label{eq:Bispest}
\Bispest =\sum_{m_{1}=-\el{1}}^{\el{1}}\sum_{m_{2}=-\el{2}}^{\el{2}}\sum_{m_{3}=-\el{3}}^{\el{3}} \wigner{\ell _{1}}{\ell _{2}} {\ell _{3}}{m_{1}} {m_{2}} {m_{3}} \
\widetilde{a}_{\ell _{1},m_{1}}\widetilde{a}_{\ell _{2},m_{2}}\widetilde{a}%
_{\ell _{3},m_{3}}.
\end{equation}%
Throughout the paper we restrict attention to ordered multipole triplets
\[
\ell_1<\ell_2<\ell_3,
\]
which avoids counting equivalent configurations because of the permutation symmetry of the bispectrum.

The following result summarizes the first four moments of the sample bispectrum (see \cite[Theorem~9.7]{MaPeCUP}).
\begin{proposition}\label{prop:Best}
Let $\widehat B_{\ell_1,\ell_2,\ell_3}$ be defined by
\eqref{eq:Bispest}, where
$\ell_1<\ell_2<\ell_3$.
Then
\begin{align}
&\Ex{ \Bispest} = \Bisp; \label{eq:exb};\\
&\Ex{\Bispest^2} = 1 + O\left(\fNL^2\right)\label{eq:varb};\\
& \cumq{\Bispest} = 6\left(\wignersix{\ell_1}{\ell_2}{\ell_3}{\ell_1}{\ell_2}{\ell_3}+\sum_{i=1}^{3}\frac{1}{2\ell_i+1} \right) + O\left(\fNL\right). \label{eq:cqb}
\end{align}
In particular,
\begin{equation*}
\cumq{\Bispest} \leq \frac{12}{2\ell_1+1}.
\end{equation*}
\end{proposition}

\begin{remark}[Asymptotics on the Third Wiener Chaos]\label{rem:chaos}
A key observation for our technical framework is the structural representation of the sample statistic within the Wiener--It\^o chaotic decomposition. Let $W$ denote a spherical white noise measure on $\sphere$, such that the frequency components of the reference Gaussian field can be expressed via the stochastic integration $T_\ell(x) = \int_{\mathbb S^2} \frac{2\ell+1}{4\pi} P_\ell(\langle x,y\rangle)\, W(\mathrm dy)$.

By the Gaunt projection identity, the sample bispectrum rewrites as an integrated spatial product:
\[
\widehat B_{\ell_1,\ell_2,\ell_3} = \begin{pmatrix} \ell_1 & \ell_2 & \ell_3\\ 0 & 0 & 0 \end{pmatrix}^{-1} h_{\ell_1\ell_2\ell_3}^{-1/2} \int_{\mathbb S^2} T_{\ell_1}(x) T_{\ell_2}(x) T_{\ell_3}(x)\, \mathrm dx.
\]
This implies that under the null Gaussian hypothesis ($f_{\mathrm{NL}}=0$), $\widehat B_{\ell_1,\ell_2,\ell_3}$ belongs strictly to the third Wiener chaos $\mathcal{H}_3$ generated by $W$ (see \cite[Lemma 9.6]{MaPeCUP}). Under the local alternative framework where $f_{\mathrm{NL}} \to 0$ as $L \to \infty$, the asymptotic behavior of the global estimator is dominated by this chaotic projection. This enables the deployment of the Stein--Malliavin calculus, specifically the Fourth Moment Theorem frameworks \cite{noupebook}, which map bounds on fourth-order cumulants directly into total variation Berry--Esseen bounds.
\end{remark}

\section{Ordinary Least Squares estimation of the nonlinearity parameter and asymptotic Gaussianity}\label{sec:OLS}

In this section, we construct the statistical estimator for the non-Gaussianity amplitude parameter $\fNL$ and establish its high-frequency asymptotic characterization. After defining the sampling scheme over a restricted multipole configuration domain, we formalize the estimation problem through an ordinary least squares (OLS) projection framework, naturally embedding the classical Komatsu--Spergel--Wandelt (KSW) statistic \cite{ksw05} within the general theory of linear models. We then derive the sharp scaling laws for its high-order asymptotic moments and deploy Stein--Malliavin calculus techniques \cite{nourdinpeccati,noupebook} on fixed Wiener chaoses to establish a quantitative central limit theorem with explicit rates of convergence in total variation distance.

We assume the observer records a finite, high-frequency bandwidth realization of the normalized harmonic coefficients:
\[
\{\widetilde a_{\ell,m}:\ell=L_0,\ldots,L,\;m=-\ell,\ldots,\ell\},
\]
are observed, where $L$ denotes the high-frequency truncation limit, and the lower cutoff $L_0$ is defined as
\[
L_0=rL,\qquad r\in(0,1/2).
\]
The parameter $r$ determines the lower cutoff of the frequency band. Keeping $r$ fixed while letting $L\to\infty$ corresponds to the standard narrow-band asymptotic framework, where inference is based exclusively on high-frequency multipoles (see, e.g., \cite{dlm}). We assume exact reconstruction of the harmonic coefficients, neglecting observational noise and aliasing effects (see, e.g., \cite{dp19}).

The third-order sample interactions are mapped across the discrete triangular configuration space $\Lambda_L$, defined as the ordered subset:
\begin{equation}\label{eq:LambdaL}
\Lambda_L=
\left\{
(\ell_1,\ell_2,\ell_3):
L_0\le \ell_1<\ell_2<\ell_3\le L,\;
\ell_3\le \ell_1+\ell_2
\right\}.
\end{equation}
The ordering removes the redundancy induced by permutation symmetry, while the triangle condition is precisely the admissibility condition for the associated Wigner $3j$-symbols.

\begin{remark}[Combinatorial Scaling on the Wigner manifold]\label{rem:Lambda}
Let $N_L = |\Lambda_L|$ denote the cardinality of the admissible triplet configuration set. By rescaling the discrete multipole index lattice by the high-frequency envelope, $x_i = \ell_i/L$, the set $\Lambda_L$ corresponds asymptotically to the bounded continuous domain $\mathcal{A}_r \subset [r,1]^3$ defined by the inequalities $r \le x_1 < x_2 < x_3 \le 1$ under the upper bounding face $x_3 \le x_1+x_2$. Computing the volume of this polytope via iterated integration yields:
\[
\int_r^1\int_{x_1}^1 \bigl(\min\{1,x_1+x_2\}-x_2\bigr)\, dx_2\,dx_1 = \frac{1}{12}-\frac{r^2}{2}+\frac{r^3}{2}.
\]
Consequently, the discrete configuration count scales as a homogeneous cubic form in $L$:
\[
N_L = \frac{(1-r)^3}{12}L^3 + o(L^3), \qquad \text{as } L\to\infty.
\]
\end{remark}

\subsection{Ordinary least squares method and nonlinearity parameter}
The expectation of the bispectrum is linear in $f_{\mathrm{NL}}$ (Proposition~\ref{prop:expect}). This justifies the formulation of the linear statistical model:
\begin{equation}\label{eq:model}
\widehat B_{\ell_1,\ell_2,\ell_3}
=
f_{\mathrm{NL}}\eta_{\ell_1,\ell_2,\ell_3}
+\varepsilon_{\ell_1,\ell_2,\ell_3},
\qquad
(\ell_1,\ell_2,\ell_3)\in\Lambda_L,
\end{equation}
where the non-stationary error terms $\varepsilon_{\ell_1,\ell_2,\ell_3}$ correspond to the stochastic fluctuations of the empirical bispectral projections around their true expectations.

\begin{remark}[Local Alternatives Framework]\label{rem:local}
To avoid asymptotic trivialization and evaluate the power of the estimator against contiguous structural alternatives, we embed the problem within a local-alternative framework. Specifically, we allow the non-Gaussian coupling constant to decay as a function of the resolution parameter:
\[
f_{\mathrm{NL}} = f_{\mathrm{NL}}(L) \longrightarrow 0, \qquad \text{as } L\to\infty.
\]
This corresponds to the classical local-alternatives framework, where departures from Gaussianity become increasingly difficult to detect as the available resolution grows. Such an assumption is also compatible with inflationary models predicting only weak primordial non-Gaussianity (see, e.g., \cite{LSFS10}).

Moreover, throughout the asymptotic analysis we regard the Gaussian field $T_G$ as the reference probability structure. Under the assumption $f_{\mathrm{NL}}(L)\to0$, the leading contribution to the sample bispectrum is its Gaussian component, which is a third-order Wiener chaos
functional. This asymptotic representation provides the probabilistic framework for the Stein--Malliavin arguments developed below.
\end{remark}

By collecting the elements $\{\widehat B_{\ell_1,\ell_2,\ell_3}\}$, $\{\eta_{\ell_1,\ell_2,\ell_3}\}$, and $\{\varepsilon_{\ell_1,\ell_2,\ell_3}\}$ over the lattice index $\Lambda_L$ into the coordinate vectors $Y, H, E \in \mathbb{R}^{N_L}$, the model expression \eqref{eq:model} simplifies to the regression equation 
\begin{equation}\label{eq:vec}
Y = \fNL H + E. 
\end{equation}
The ordinary least squares estimator is then defined by minimizing the standard quadratic loss function on the discrete manifold:
\[
\fNLest = \argmin_{f\in\mathbb R} \|Y-fH\|^2 = (H^TH)^{-1}H^TY.
\]
Component-wise, this linear projection reads:
\begin{equation}\label{eq:estimatordef}
\widehat f_{\mathrm{NL}} = \frac{ \sum_{(\ell_1,\ell_2,\ell_3)\in\Lambda_L} \eta_{\ell_1,\ell_2,\ell_3} \widehat B_{\ell_1,\ell_2,\ell_3} }{ \sum_{(\ell_1,\ell_2,\ell_3)\in\Lambda_L} \eta_{\ell_1,\ell_2,\ell_3}^2 }.
\end{equation}%
While \eqref{eq:estimatordef} is structurally equivalent to the KSW estimator optimized for fast Fourier pipelines \cite{ksw05}, its formalization as a coordinate OLS estimator over the Wigner lattice allows for direct deployment of spatial regression limit theorems and simplifies the moment bounding procedure.

\subsection{Asymptotic properties of the estimator}\label{sec:main}
We now analyze the structural behavior of $\fNLest$ as $L \to \infty$. We first establish continuous integral approximations for the lattice weights, which provide exact scaling rates for the variance and the high-order cumulants. Then, we derive a quantitative central limit theorem under the total variation distance.

\subsubsection{Integral representations and spectral scaling laws}

The continuous limit of the lattice structure requires tracking the geometry of the triangle support. Fix $r\in(0,1/2)$, and let $\mathcal{A}_r$ be the continuous simplex:
\begin{equation}\label{eq:intdomain}
	\Ar = \lbrace \left(x_1,x_2,x_3\right) \in \mathbb{R}^3, r\leq x_1<x_2<x_3 \leq 1; x_2-x_1 \leq x_3 \leq x_1+x_2 \rbrace,
\end{equation}
We define the symmetric polynomial kernel $\Delta: \mathbb{R}^3 \to \mathbb{R}$ as:
\begin{equation}\label{eq:delta}
	\Delta\left(x_1,x_2,x_3\right)=\left(x_1+x_2+x_3\right)\left(x_1+x_2-x_3\right)\left(x_1-x_2+x_3\right)\left(-x_1+x_2+x_3\right).
\end{equation}
By Heron's formula, $\Delta(x_1,x_2,x_3) = 16T^2(x_1,x_2,x_3)$, where $T$ is the area of the continuous flat triangle defined by the side lengths $x_1,x_2,x_3$.

Our first technical lemma establishes an exact reduction principle for three-dimensional spatial integrals over the triangular manifold $\mathcal{A}_r$.
\begin{lemma}\label{lemma:integral}
	For $(a,b,c,d) \in \mathbb{R}$, and $r \in \left(0,1/2\right)$, let $D_r(a,b,c,d)$ be given by 
	\begin{equation}\label{eq:genint}
		D_r(a,b,c,d) = \int_{\Ar} \frac{x_1^a x_2^bx_3^c}{\Delta^d(x_1,x_2,x_3)} \diff x_1 \diff x_2 \diff x_3,
	\end{equation}
	where $\Ar$ and $\Delta:\mathbb{R}^3\mapsto \mathbb{R}$ are given by \eqref{eq:intdomain} and \eqref{eq:delta} respectively. Let $\lambda=a+b+c-4d+3$. Then it holds that 
	\begin{align}\label{eq:Gfinale}
		D_r(a,b,c,d) = \begin{cases}
			\frac{1}{\lambda} \int_{\Omega_r} \frac{u^av^b}{Q^d(u,v)}\left(1-\left(\frac{r}{u}\right)^\lambda\right) 	\diff u \diff v& \text{ if }\lambda \neq 0		 \\
		\int_{\Omega_r} \frac{u^av^b}{Q^d(u,v)}\log\left(\frac{r}{u}\right) \diff u \diff v	& \text{ if }\lambda = 0		
			\end{cases},
	\end{align}
where $Q: \mathbb{R} \mapsto \mathbb{R}$ is given by 
\begin{equation*}\label{eq:Q}
	Q\left(u,v\right)=\Delta\left(u,v,1\right)= \left(u+v+1\right) \left(u+v-1\right) \left(u-v+1\right) \left(-u+v+1\right),
\end{equation*}
and the integration domain is defined by
\begin{equation}\label{eq:omegadomain}
	\begin{split}
&\Omega_r=\Omega_{1,r} \cup \Omega_{2,r}; \\ 
& \Omega_{1,r} =\lbrace (u,v): r<u<\frac{1}{2}, 1-u<v<1\rbrace;\\
&  \Omega_{2,r} =\lbrace (u,v): \frac{1}{2}<u<1, u<v<1\rbrace.
\end{split}
\end{equation} 
\end{lemma}

To link the continuous domain with the algebraic weights, we leverage the uniform asymptotic expansion for the Wigner $3j$-symbols on the boundary of the classical regime (see, e.g., \cite{vmk}).
\begin{lemma}
	\label{LemmaFlashGordon} For admissible triples satisfying the triangle conditions, as $\min\{\ell_1,\ell_2,\ell_3\}\to\infty$, it holds that
	\begin{equation}
		\wigner{\ell _{1}} {\ell _{2}} { \ell _{3} }{0} {0} {0}
		\approx (-1)^{\frac{\ell_1+\ell_2+\ell_3}{2}}\sqrt{\frac{2}{\pi }}\Delta\left( \ell _{1},\ell _{2},\ell _{3}\right) ^{-\frac{1}{4}} \label{FlashGordon1}
	\end{equation}%
	where $	\Delta$ is defined as in  \eqref{eq:delta}.%
\end{lemma}

By combining the continuous integral representation (Lemma~\ref{lemma:integral}) with the scaling of the coupling symbols (Lemma~\ref{LemmaFlashGordon}), we obtain the exact scaling behaviors of the normalization and the higher-order weight vectors.

\begin{lemma}\label{lemma:eta}
	Let $\eta _{\ell _{1}\ell _{2}\ell _{3}}$ be given by \eqref{eq:eta} and $\kappal$ be defined as
	\begin{equation*}\label{eq:kappa}
		\kappal=\etal^4 (2\ell_1+1)^{-1}.
	\end{equation*}
	Let $\Seta$ and $\Skappa$ be given by 
	\begin{align*}
		& \Seta = \summinus \etal ^2;\\
		& \Skappa=  \summinus \kappal, 
	\end{align*}
	where $\LambdaL$ is defined by \eqref{eq:LambdaL}. It holds that 
	\begin{align*}
		&\underset{L \rightarrow \infty}{\lim} L^{\alpha-4} \Seta = C_{\eta^2} I_{\eta^2;r}\left( \alpha\right);\\
		&\underset{L \rightarrow \infty}{\lim} L^{2\left(\alpha-2\right)} \Skappa = C_{\kappa} I_{\kappa;r}\left( \alpha\right),
	\end{align*}
	where
	\begin{equation*}
		\begin{split}
			&	C_{\eta^2}= \frac{A}{\pi^4}; \quad 	I_{\eta^2;r} \left(\alpha\right)= D_r\left(1-\alpha,1-\alpha,1+\alpha,\frac{1}{2}\right);\\ &C_\kappa = \frac{6A^2}{\pi^8}; \quad I_{\kappa;r} \left(\alpha\right)= D_r\left(1-2\alpha,2\left(1-\alpha\right),2\left(1+\alpha\right),1\right).
		\end{split}
	\end{equation*}
\end{lemma}

\subsubsection{Asymptotic moments of the estimator}

The asymptotic approximations established in Lemma~\ref{lemma:eta} allow us to determine the leading-order behavior of the first four moments of the estimator $\fNLest$. In particular, the following proposition shows that the ordinary least squares estimator is asymptotically unbiased, identifies the asymptotic variance, and derives an upper bound for its fourth cumulant. These results provide the quantitative ingredients required for the Stein--Malliavin normal approximation developed below.

\begin{proposition}\label{lemmaest}
Let $\widehat f_{\mathrm{NL}}$ be the estimator defined in
\eqref{eq:estimatordef}. Under the assumptions introduced above, for $\alpha >4$, 
	it holds that 
	\begin{align*}
		&\underset{L \rightarrow \infty }{\lim}  \Ex{ \fNLest}=\fNL ;\\
		&\underset{L \rightarrow \infty }{\lim} L^{4-\alpha} \Var{ \fNLest} =\sigma_{\fNL}^2;\label{eq:varf}\\
		&0<\underset{L \rightarrow \infty }{\lim} L^{2\alpha-12} \cumq{\fNLest} \leq K_{\fNL} ,
	\end{align*}
	where 
	\begin{equation*}
		\begin{split}
		&	\sigma_{\fNL}^2 = \frac{1}{C_{\eta^2}I_{\eta^2;r} \left(\alpha\right)};\\ 
		&K_{\fNL}=\frac{C_{\kappa}I_{\kappa;r}\left(\alpha\right)}{C_{\eta^2}^4 I^4_{\eta^2;r}}.
		\end{split}
	\end{equation*}
\end{proposition}

\subsubsection{Quantitative Gaussian approximation}
By leveraging the chaotic characterization outlined in Remark~\ref{rem:chaos}, the centered estimator $\fNLest - \fNL$ can be analyzed as a linear combinations of third-order sample moments, placing it within the third Wiener chaos $\mathcal{H}_3$ associated with the reference measure $\mathbb{P}_G$. 

To establish a central limit theorem with explicit rates of convergence, we evaluate the distance between the distribution of the normalized estimator and a standard Gaussian variable $Z \sim \mathcal{N}(0,1)$ in the total variation metric. Recall that for any two probability measures $P$ and $Q$ defined on $(\Omega, \mathcal{F})$, the total variation distance is defined as $d_{\mathrm{TV}}(P,Q) = \sup_{A\in\mathcal{F}} |P(A)-Q(A)|$.

If $F$ belongs to the $q$-th Wiener chaos and $Z\sim\mathcal N(0,1)$, following \cite[Theorem~5.2.6]{noupebook}, the total variation distance between the law of $F$ and the standard normal distribution is bounded by a function of the fourth cumulant of $F$:
\[
d_{\mathrm{TV}}(F,Z)
\le
2\sqrt{\frac{q-1}{3q}
\frac{\operatorname{cum}_4(F)}
{\operatorname{Var}(F)^2}}.
\]
Since $\widehat f_{\mathrm{NL}}-\fNL$ belongs to the third Wiener chaos ($q=3$), we obtain
\[
d_{\mathrm{TV}}
\left(
\frac{\widehat f_{\mathrm{NL}}-\fNL}{\sqrt{\operatorname{Var}(\widehat f_{\mathrm{NL}})}},
Z
\right)
\le
\frac{2\sqrt2}{3}
\sqrt{
\frac{\operatorname{cum}_4(\widehat f_{\mathrm{NL}})}
{\operatorname{Var}(\widehat f_{\mathrm{NL}})^2}
},
\]
where $Z\sim\mathcal N(0,1)$.

Substituting the precise high-frequency scaling trajectories from Proposition~\ref{lemmaest} into this inequality allows us to balance the polynomial orders in $L$. This yields the following quantitative central limit theorem.
\begin{theorem}\label{thm:main}
Under the assumptions of Proposition~\ref{lemmaest}, it holds that 
\begin{equation}\label{eq:mainth}
	\dW\left(\frac{\widehat f_{\mathrm{NL}}-\fNL}{\sqrt{\operatorname{Var}(\widehat f_{\mathrm{NL}})}}, Z \right) \leq 
	\CW L^{-2},
\end{equation}
where 
\begin{equation*}
\CW=\frac{4}{\sqrt{3}}\frac{I^{\frac{1}{2}}_{\kappa;r}\left(\alpha\right)}{I_{\eta^2;r}\left(\alpha\right)}. 
\end{equation*}
Also, as $L\rightarrow \infty$, $\fNLest-\fNL$ converges in distribution to a standard normal random variable, that is,
\begin{equation*}
	\frac{\widehat f_{\mathrm{NL}}-\fNL}{\sqrt{\operatorname{Var}(\widehat f_{\mathrm{NL}})}} \xrightarrow{d} \mathcal{N}(0,1)
\end{equation*}
\end{theorem}

\begin{remark}[Asymptotic Rate and Geometric Boundary Phenomena]\label{rem:rates}
\begin{enumerate}
\item \textit{Configuration-Space Interpretation:} The algebraic rate $O(L^{-2})$ established in Theorem~\ref{thm:main} can be interpreted in terms of the number of available independent bispectrum configurations. From Remark~\ref{rem:Lambda}, the cardinality of the configuration space grows cubically, $N_L = O(L^3)$. Rewriting the total variation error bound in terms of sample size yields:
\[
d_{\mathrm{TV}}\left(\frac{\widehat f_{\mathrm{NL}}-f_{\mathrm{NL}}}{\sqrt{\Var{\widehat f_{\mathrm{NL}}}}}, Z \right) = O(N_L^{-2/3}).
\]
This convergence is strictly faster than the classical Berry--Esseen bound $O(N_L^{-1/2})$ typical of independent identically distributed sums. This acceleration highlights the regularizing behavior of the Wigner $3j$-coupling weights, which suppress the strong cross-dependencies among overlapping triads on the sphere.
\item \textit{Chaotic Universality and Critical Regimes:} The appearance of the $2/3$ exponent mirrors results found in non-local central limit theorems on the third Wiener chaos. For instance, in the study of non-linear functionals of stationary long-range dependent sequences \cite{npopt}, normalized Hermite forms $F_N = \frac{1}{\sqrt{N}}\sum H_3(X_k)$ under a critical covariance decay profile $\rho(j) = O (|j|^{-2/3})$ yield an identical $O(N^{-2/3})$ total variation decay rate. This suggests that the KSW algebraic filtering structure achieves an optimal balancing on the underlying Fock space.
\item \textit{Lattice Discrepancy Analogy:} A structural explanation for this convergence rate stems from spatial discrepancy theory. In the classical Gauss circle problem, the number of integer lattice points contained within a planar domain of scale $R$ deviates from its continuous area by a boundary error $\varepsilon(R)$. The relative approximation error scales as $O(R^{-4/3}) = O(N^{-2/3})$, where $N=O(R^2)$ matches the interior lattice count \cite{hardy79}. Because the proof of Theorem~\ref{thm:main} relies on approximating the discrete sum over the Wigner lattice $\Lambda_L$ with a continuous integral over the simplex $\mathcal{A}_r$, the $O(L^{-2}) = O(N_L^{-2/3})$ rate can be understood as an algebraic reflection of this boundary discrepancy phenomenon.
\end{enumerate}
\end{remark}

\section{Numerical illustration}\label{sec:numerics}

\begin{figure}[!htbp]
\centering
\includegraphics[width=.72\textwidth]{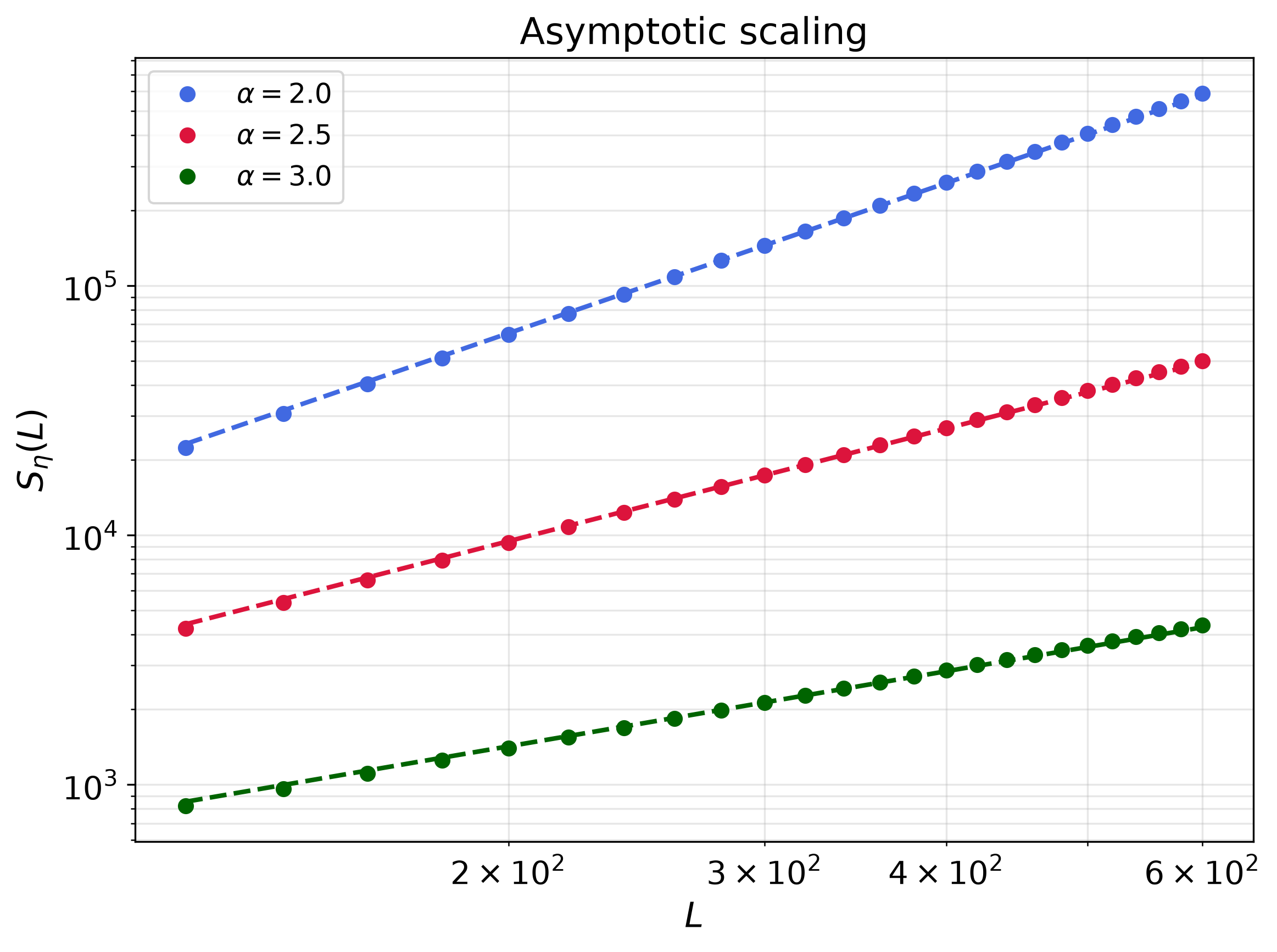}
\caption{
Log--log representation of the discrete sums $S_{\eta}(L)$ for
$\alpha=2$, $2.5$, and $3$ with $r=0.2$. The approximately linear behaviour confirms the predicted power-law scaling
$S_{\eta}(L)\propto L^{4-\alpha}$.
}
\label{fig:scaling}
\end{figure}

In this section, we present a systematic numerical investigation designed to evaluate the empirical validity of the high-frequency scaling laws and the quantitative normal approximations established in Lemma~\ref{lemma:eta} and Theorem~\ref{thm:main}. Given that our main results hold asymptotically, these simulations are intended solely for illustrative purposes.

We first examine the discrete weight sum over the triangular manifold $\Lambda_L$:
\[
S_{\eta}(L)=\sum_{\Lambda_L}
\frac{\ell_1^{1-\alpha}\ell_2^{1-\alpha}\ell_3^{1+\alpha}}
{\Delta(\ell_1,\ell_2,\ell_3)^{1/2}},
\]
with lower cutoff $r=0.2$, and compare its growth for increasing values of $L$ with the theoretical prediction
\[
S_{\eta}(L)=O(L^{4-\alpha}).
\]

Figure~\ref{fig:scaling} illustrates the empirical trajectories of $S_{\eta}(L)$ on a logarithmic scale for three representative spectral decay configurations: $\alpha=2.0$, $\alpha=2.5$, and $\alpha=3.0$. Across all examined regimes, the discrete configurations line up precisely along linear trajectories, validating the power-law exponent predicted by the continuous Riemannian approximation on the Wigner lattice.

\begin{table}[!htbp]
\centering
\caption{Estimated scaling exponents obtained by linear regression on the log--log plots.}
\label{tab:scaling}
\begin{tabular}{ccc}
\hline
$\alpha$ & Theory $(4-\alpha)$ & Estimated exponent\\
\hline
2.0 & 2.000 & 2.0297\\
2.5 & 1.500 & 1.5331\\
3.0 & 1.000 & 1.0365\\
\hline
\end{tabular}
\end{table}

To establish a formal quantitative benchmark for this alignment, Table~\ref{tab:scaling} reports the estimated scaling exponents obtained via ordinary least squares regression of $\log S_{\eta}(L)$ against $\log L$. The empirical exponents closely match the analytical values $4-\alpha$, with relative discrepancies remaining strictly bounded below $4\%$ across all testing configurations. Notably, the estimation accuracy increases for lower values of the spectral index $\alpha$, where the low-order power-law decay allows the system to enter the high-frequency regime at lower truncation thresholds.

Overall, the empirical scaling confirms that $S_\eta(L) = O(L^{4-\alpha})$, demonstrating that the continuous sieve approximation developed in Lemma~\ref{lemma:integral} accurately captures the geometric discrete-to-continuous boundary transitions even under moderate, non-asymptotic bandwidth horizons.

\begin{figure}[!htbp]
\centering
\includegraphics[width=\textwidth]{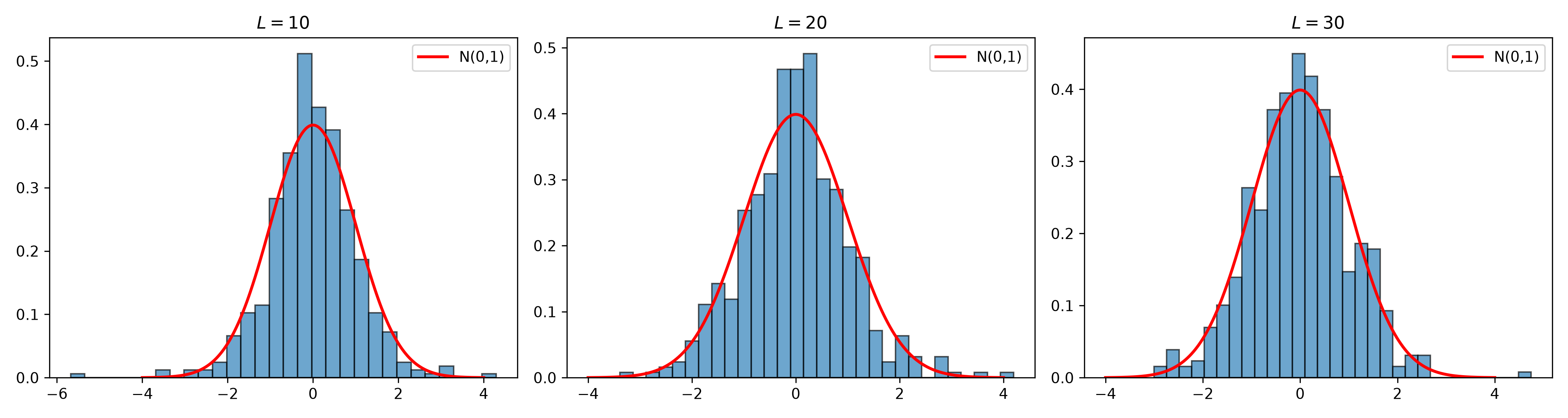}
\caption{Empirical distributions of the standardized estimator
$\widehat f_{\mathrm{NL}}$ under the Gaussian null hypothesis for
$L=10,20,30$, based on $1000$ Monte Carlo realizations. The red curve
corresponds to the standard normal density. The agreement with the
Gaussian approximation improves as the maximum multipole increases.}
\label{fig:histograms}
\end{figure}

To validate the quantitative central limit theorem established in Theorem~\ref{thm:main}, we performed a comprehensive Monte Carlo simulation study under the Gaussian null hypothesis $f_{\mathrm{NL}}=0$. For each discrete resolution level $L \in \{10, 20, 30\}$, we simulated $1\,000$ independent pseudo-random realizations of the studentized spherical harmonic coefficients $\{\widetilde a_{\ell m}\}$ satisfying the covariance structure \eqref{eq:powerdef}. For each iteration, we evaluated the OLS non-Gaussianity estimator $\widehat f_{\mathrm{NL}}$ defined in \eqref{eq:estimatordef}, and centered and standardized the resulting empirical sample to enforce mean zero and unit variance.

Figure~\ref{fig:histograms} displays the empirical probability density functions of the studentized estimator contrasted against the standard Gaussian reference density $\mathcal{N}(0,1)$. As the high-frequency envelope $L$ expands, the empirical profiles show a clear regularizing trend, rapidly matching the Gaussian target distribution as predicted by the Stein--Malliavin bounds.

\begin{table}[!htbp]
\centering
\caption{Summary statistics for the standardized estimator
$\widehat f_{\mathrm{NL}}$ under the Gaussian null hypothesis, based on
1000 Monte Carlo realizations.}
\label{tab:gaussian}
\begin{tabular}{lccc}
\toprule
$L$ & Skewness & Excess kurtosis & Shapiro--Wilk $p$-value\\
\midrule
10 & -0.2393 & 2.9345 & $6.82\times10^{-8}$\\
20 & \phantom{-}0.2033 & 0.9941 & $1.32\times10^{-3}$\\
30 & \phantom{-}0.1159 & 0.8327 & $1.72\times10^{-2}$\\
\bottomrule
\end{tabular}
\end{table}

This chaotic regularizing mechanism is statistically monitored in Table~\ref{tab:gaussian}, which tracks the evolution of the empirical sample skewness, sample excess kurtosis, and the corresponding $p$-values from the Shapiro--Wilk normality test. While the low-resolution regime ($L=10$) exhibits noticeable non-Gaussian signatures, characterized by a localized left-skewness and an inflated tail behavior, the high-order cumulants show a systematic decay as the resolution expands. By the time the bandwidth reaches $L=30$, both the skewness and excess kurtosis contract substantially, indicating a fast clearance of the residual third- and fourth-order cumulants from the underlying Wiener chaos projection. Although the Shapiro--Wilk test still detects the structural discreteness of the lattice boundary at these moderate resolution levels, the diagnostic trajectories move monotonically toward their target Gaussian limits. These controlled finite-sample experiments thus provide empirical evidence of the fast Central Limit Theorem trends established by our technical analysis.

\section{Proofs}\label{sec:proofs}
\subsection{Preliminary results}
The results in this subsection are classical, and the proofs are reproduced here in abbreviated form from \cite{m,MaPeCUP}.

\begin{proof}[Proof of Proposition \ref{prop:almtwo}]
	Using the identity $\alm =(-1)^m\cc{a}_{\ell,-m}$ (see for instance \cite{MaPeCUP}) together with Equation \eqref{eq:powerdef} we obtain
	\begin{align*}
		\mathbb{E}\left[\almtwo\right]& = \sum_{\ell _{1},\ell _{2}\geq 1}  \sum_{m_{1},m_{2}}   \mathbb{E}\left[a_{\ell_1,m_1;G} a_{\ell_2,m_2;G} \right] \sqrt{\frac{\left(\lcardbis{1}\right)\left(\lcardbis{2}\right)}{4\pi \left(\lcard\right)}}   \cg \cgnull\\
		& = \sum_{\ell _{1}\geq 1}  \sum_{m_{1}=-\ell_1}^{\ell_1}(-1)^{m_1}C_{\ell_1;G} \frac{\left(\lcardbis{1}\right)}{\sqrt{4\pi \left(\lcard\right)}}  \cgbis{\ell}{0}{\ell_1}{m_1}{\ell_1}{-m_1} \cgbis{\ell}{0}{\ell_1}{0}{\ell_1}{0},
	\end{align*}
	where the last equality follows $m=m_1-m_1=0$ in view of Equation \eqref{eq:m}. Applying \eqref{eq:prop1} and \eqref{eq:prop2} yields 
	\begin{align*}
		\mathbb{E}\left[\almtwo\right]
		&  = \delta_{\ell}^{0}\delta_{m}^{0}\sum_{\ell _{1}\geq 0} (-1)^{\ell_1} C_{\ell_1;G} \left[\frac{\left(\lcardbis{1}\right)^3}{4\pi } \right]^{\frac{1}{2}}  \cgbis{0}{0}{\ell_1}{0}{\ell_1}{0} =0,
	\end{align*}
	for $\ell\geq 1$. 
Hence, 
\begin{align*}
			\Ex{H_2(T_G(x))} &= 0. 
\end{align*}
We next compute the covariance matrix. For each pair $\ell,m$ and $\ell^\prime,m^\prime$, it holds that
	\begin{align}\notag
		\Ex{\almtwo \cc{a}_{\ell^\prime,m^\prime;2}} &=\sum_{\ell _{1},\ell _{2},\ell _{1}^\prime,\ell _{2}^\prime} \sum_{m_{1},m_{2},m_{1}^\prime,m_{2}^\prime} \sqrt{\frac{\left(\lcardbis{1}\right)\left(\lcardbis{2}\right)\left(2\ell_1^\prime+1\right)\left(2\ell_2^\prime+1\right)}{\left(\lcard\right)\left(2\ell^\prime+1\right)\left(4\pi\right)^2}}\\
		&\notag \times \cg \cgnull \cgbis{\ell^\prime}{m^\prime}{\ell_1^\prime}{m_1^\prime}{\ell_2^\prime}{m_2^\prime}\cgbis{\ell^\prime}{0}{\ell_1^\prime}{0}{\ell_2^\prime}{0}\\
		&\times  \Ex{a_{\ell_1,m_1;G}a_{\ell_2,m_2;G}\cc{a}_{\ell_1^\prime,m_{1}^\prime;G} \cc{a}_{\ell_{2}^\prime,m_{2}^\prime;G}}.  \label{eq:covalmtwo}
	\end{align}
	Since the harmonic coefficients are Gaussian, Wick's theorem yields
	\begin{align*}
		\Ex{a_{\ell_1,m_1;G}a_{\ell_2,m_2;G}\cc{a}_{\ell_1^\prime,m_{1}^\prime;G} \cc{a}_{\ell_{2}^\prime,m_{2}^\prime;G}} &=   \Ex{a_{\ell_1,m_1;G} \cc{a}_{\ell_1^\prime,m_{1}^\prime;G}} \Ex{a_{\ell_2,m_2;G} \cc{a}_{\ell_{2}^\prime,m_{2}^\prime;G}} \\
		&+  \Ex{a_{\ell_1,m_1;G} \cc{a}_{\ell_2^\prime,m_{2}^\prime;G}} \Ex{a_{\ell_2,m_2;G} \cc{a}_{\ell_{1}^\prime,m_{1}^\prime;G}} \\
		& + (-1)^{m_2+m_2^\prime} \Ex{a_{\ell_1,m_1;G} \cc{a}_{\ell_2,-m_{2};G}} \Ex{ \cc{a}_{\ell_{1}^\prime,m_{1}^\prime;G}a_{\ell_{2}^\prime,-m_{2}^\prime;G}} \\
		&= C_{\ell_1;G}C_{\ell_2;G}\delta_{\ell_1}^{\ell_1^\prime}\delta_{m_1}^{m_1^\prime}\delta_{\ell_2}^{\ell_2^\prime}\delta_{m_2}^{m_2^\prime} +C_{\ell_1;G}C_{\ell_2;G}\delta_{\ell_1}^{\ell_2^\prime}\delta_{m_1}^{m_2^\prime}\delta_{\ell_2}^{\ell_1^\prime}\delta_{m_2}^{m_1^\prime}\\
		&+ (-1)^{m_1+m_1^\prime} C_{\ell_1;G}C_{\ell_1^\prime;G}\delta_{\ell_1}^{\ell_2}\delta_{m_1}^{m_2}\delta_{\ell_1^\prime}^{\ell_2^\prime}\delta_{m_1^\prime}^{m_2^\prime}\\
		&=C_1+C_2+C_3.
	\end{align*}
We decompose \eqref{eq:covalmtwo} according to these three contributions.

The contribution of $C_1$ leads to
	\begin{align*}
		&\sum_{\ell _{1},\ell _{2},\ell _{1}^\prime,\ell _{2}^\prime} \sum_{m_{1},m_{2},m_{1}^\prime,m_{2}^\prime} \sqrt{\frac{\left(\lcardbis{1}\right)\left(\lcardbis{2}\right)\left(2\ell_1^\prime+1\right)\left(2\ell_2^\prime+1\right)}{\left(\lcard\right)\left(2\ell^\prime+1\right)\left(4\pi\right)^2}} \\
		& \times \cg \cgnull \cgbis{\ell^\prime}{m^\prime}{\ell_1^\prime}{m_1^\prime}{\ell_2^\prime}{m_2^\prime}\cgbis{\ell^\prime}{0}{\ell_1^\prime}{0}{\ell_2^\prime}{0}C_{\ell_1;G}C_{\ell_2;G}\delta_{\ell_1}^{\ell_1^\prime}\delta_{m_1}^{m_1^\prime}\delta_{\ell_2}^{\ell_2^\prime}\delta_{m_2}^{m_2^\prime}  \\
		&\qquad \qquad =\delta_{\ell}^{\ell^\prime}\delta_{m}^{m^\prime}\sum_{\ell _{1},\ell _{2}}  \frac{\left(\lcardbis{1}\right)\left(\lcardbis{2}\right)}{4\pi\left(\lcard\right)} \left(\cgnull\right)^2C_{\ell_1;G}C_{\ell_2;G},
	\end{align*}
	where in the last equality we used Equation \eqref{eq:prop3}.

	The contribution of $C_2$, combined with Equation \eqref{eq:prop0}, yields
	\begin{align*}
		&\sum_{\ell _{1},\ell _{2},\ell _{1}^\prime,\ell _{2}^\prime} \sum_{m_{1},m_{2},m_{1}^\prime,m_{2}^\prime} \sqrt{\frac{\left(\lcardbis{1}\right)\left(\lcardbis{2}\right)\left(2\ell_1^\prime+1\right)\left(2\ell_2^\prime+1\right)}{\left(\lcard\right)\left(2\ell^\prime+1\right)\left(4\pi\right)^2}} \\
		& \times \cg \cgnull \cgbis{\ell^\prime}{m^\prime}{\ell_1^\prime}{m_1^\prime}{\ell_2^\prime}{m_2^\prime}\cgbis{\ell^\prime}{0}{\ell_1^\prime}{0}{\ell_2^\prime}{0}C_{\ell_1;G}C_{\ell_2;G}\delta_{\ell_1}^{\ell_2^\prime}\delta_{m_1}^{m_2^\prime}\delta_{\ell_2}^{\ell_1^\prime}\delta_{m_2}^{m_1^\prime}  \\
		&\qquad \qquad =\delta_{\ell}^{\ell^\prime}\delta_{m}^{m^\prime}\sum_{\ell _{1},\ell _{2}}  \frac{\left(\lcardbis{1}\right)\left(\lcardbis{2}\right)}{4\pi\left(\lcard\right)} \left(\cgnull\right)^2C_{\ell_1;G}C_{\ell_2;G}.
	\end{align*}

Applying again \eqref{eq:prop1} and \eqref{eq:prop2}, the contribution of $C_3$ vanishes for $\ell>0$:
	\begin{align*}
		&\sum_{\ell _{1},\ell _{2},\ell _{1}^\prime,\ell _{2}^\prime} \sum_{m_{1},m_{2},m_{1}^\prime,m_{2}^\prime} \sqrt{\frac{\left(\lcardbis{1}\right)\left(\lcardbis{2}\right)\left(2\ell_1^\prime+1\right)\left(2\ell_2^\prime+1\right)}{\left(\lcard\right)\left(2\ell^\prime+1\right)\left(4\pi\right)^2}} \\
		& \times \cg \cgnull \cgbis{\ell^\prime}{m^\prime}{\ell_1^\prime}{m_1^\prime}{\ell_2^\prime}{m_2^\prime}\cgbis{\ell^\prime}{0}{\ell_1^\prime}{0}{\ell_2^\prime}{0}(-1)^{m_1+m_1^\prime} C_{\ell_1;G}C_{\ell_1^\prime;G}\delta_{\ell_1}^{\ell_2}\delta_{m_1}^{-m_2}\delta_{\ell_1^\prime}^{\ell_2^\prime}\delta_{m_1^\prime}^{-m_2^\prime} \\
		&= \delta_{m}^{0}\delta_{m^\prime}^{0}\sum_{\ell _{1}\ell _{1}^\prime}  (-1)^{\ell_1+\ell_1^\prime}\frac{\left(\left(2\ell_1+1\right)\left(2\ell_1^\prime+1\right)\right)^{\frac{3}{2}}}{4\pi \sqrt{\left(\lcard\right)\left(2\ell^\prime+1\right)}} 
		\cgbis{0}{0}{\ell_1}{0}{\ell_1}{0}\cgbis{0}{0}{\ell^\prime_1}{0}{\ell^\prime_1}{0}C_{\ell_1;G}C_{\ell_1^\prime;G}	= 0,
	\end{align*}
	for $\ell >0$. Summing the contributions of $C_1$, $C_2$ and $C_3$ yields
	\begin{align*}
		&\Ex{\almtwo \cc{a}_{\ell^\prime,m^\prime;2}}-\Ex{\almtwo}\Ex{ {a}_{\ell^\prime,m^\prime;2}}
		=C_{\ell;2} \delta^{\ell}_{\ell^{\prime}}\delta^{m}_{m^{\prime}}.
	\end{align*}
	For $x,y \in \sphere$, the harmonic expansion of the covariance function then follows from the addition formula for spherical harmonics, yielding
	\begin{equation*}
		\Cov{H_2\left((T_G(x)\right)}{H_2\left(T_G(y)\right)}=\sum_{\ell \geq 1} C_{\ell;2} \frac{2\ell+1}{4\pi} P_{\ell}\left(\langle x,y\rangle\right).
	\end{equation*}
\end{proof}

\begin{proof}[Proof of Proposition \ref{prop:complete}]
	First, using Proposition \ref{prop:almtwo} yields 
	\begin{align*}
		\Ex{\alm} & = \Ex{\almG+\fNL\almtwo}\\
		&=  \Ex{\almG}+ \fNL\Ex{\almtwo}\\
		&=0.
	\end{align*}
	Then, for any $x \in \sphere$, we have that 
	\begin{align*}
		\Ex{T(x)} & = 0 .
	\end{align*}
	Also, using Proposition~\ref{prop:almtwo} and recalling that $\Ex{a_{\ell_1,m_1;G}a_{\ell_2,m_2;G}a_{\ell_3,m_3;G}}=0$ for any possible choice of $\ell_1,\ell_2,\ell_3,m_1,m_2,m_3$ leads to:
	\begin{align*}
		\Ex{\alm \cc{a}_{\ell^\prime,m^\prime}} & = \Ex{\left(a_{\ell,m;G}+\fNL a_{\ell,m;2}\right)  \left(\cc{a}_{\ell^\prime,m^\prime;G}+\fNL\cc{a}_{\ell^\prime,m^\prime;2}\right)}\\
		&= \Ex{a_{\ell,m;G}\cc{a}_{\ell^\prime,m^\prime;G}}  + \fNL^2 \Ex{a_{\ell,m;2}\cc{a}_{\ell^\prime,m^\prime;2}},
	\end{align*}
	so that
	\begin{align*}	
		\Ex{\alm \cc{a}_{\ell^\prime,m^\prime}} -\Ex{\alm}\Ex{a_{\ell^\prime,m^\prime}} &=\left(C_{\ell;G}+\fNL^2 C_{\ell;2} \right) \delta_{\ell}^{\ell^\prime}\delta_{m}^{m^\prime}\\
		&=C_{\ell}\delta_{\ell}^{\ell^\prime}\delta_{m}^{m^\prime}.
	\end{align*}
	Indeed, since \(a_{\ell,m;2}\) is quadratic in the Gaussian coefficients, the mixed moments
\[
\mathbb E[a_{\ell,m;G}\overline{a_{\ell',m';2}}]
\quad\text{and}\quad
\mathbb E[a_{\ell,m;2}\overline{a_{\ell',m';G}}]
\]
are expectations of products of three Gaussian harmonic coefficients and therefore vanish by Wick's theorem.
	
Therefore, for $x,y\in\sphere$, it holds that
	\begin{align*}
		\Cov{T(x)}{T(y)} = \sum_{\ell \geq 1}  \frac{2\ell+1}{4\pi} C_{\ell}P_{\ell}\left(\langle x,y \rangle\right).
	\end{align*}
\end{proof}

\begin{proof}[Proof of Proposition \ref{prop:expect}]
	Following the argument of \cite{m}, substituting \eqref{eq:wT} into the definition of the bispectrum yields
	\begin{equation*}\label{eq:defprod}
		\begin{split}
&\qquad			\mathbb{E}\left[ \widetilde{a}_{\ell _{1},m_{1}}\widetilde{a}_{\ell_{2},m_{2}}\widetilde{a}_{\ell _{3},m_{3}}\right] \\
			&=\frac{1}{\sqrt{C_{\ell_1;G}C_{\ell_2;G}C_{\ell_3;G}}}\left[\Ex{ a_{\ell _{1},m_{1};G}a_{\ell _{2},m_{2};G}a_{\ell _{3},m_{3};G} } \right.\\
			&+ \fNL \left( \Ex{ a_{\ell _{1},m_{1};G}a_{\ell _{2},m_{2};G}a_{\ell _{3},m_{3};2} }\right.+\Ex{ a_{\ell _{1},m_{1};G}a_{\ell _{2},m_{2};2}a_{\ell _{3},m_{3};G} } \\
			& +\left.  \Ex{ a_{\ell _{1},m_{1};2}a_{\ell _{2},m_{2};G}a_{\ell _{3},m_{3};G} } \right) + o\left(\fNL\right). 
		\end{split}%
	\end{equation*}
Since the Gaussian harmonic coefficients have vanishing odd moments, it follows that
	\begin{equation*}
		\mathbb{E}\left[ a_{\ell _{1},m_{1};G}a_{\ell _{2},m_{2};G}a_{\ell _{3},m_{3};G}\right] =0.
	\end{equation*}%
	We now consider the three terms proportional to $\fNL$ and consider, for example, the following
	\begin{align*}
		\mathbb{E}\left[ a_{\ell _{1},m_{1};G}a_{\ell _{2},m_{2};G}a_{\ell _{3},m_{3};2}\right]&  =\sqrt{\frac{1}{\prod_{i=1}^3C_{\ell_i;G}}} \sum_{\lambda_1,\lambda_2} \sum_{\mu_1,\mu_2} \sqrt{\frac{(2\lambda_1+1)(2\lambda_2+1)}{4\pi(2\ell_3+1)}}\\
		& \times \Ex{ a_{\lambda_1,\mu_1;G}a_{\lambda_2,\mu_2;G} a_{\ell_1,m_1;G} a_{\ell_2,m_2;G}} \cgbis{\ell_3}{0}{\lambda_1}{0}{\lambda_2}{0}\cgbis{\ell_3}{m_3}{\lambda_1}{\mu_1}{\lambda_2}{\mu_2}.
	\end{align*}
	Since the harmonic coefficients are Gaussian, Wick's theorem gives us the expectation of the product of four harmonic coefficients. In particular, using the following formula (see \cite{m}):
	\begin{align*}
		&\Ex{ a_{\lambda_1,\mu_1;G}a_{\lambda_2,\mu_2;G} a_{\ell_1,m_1;G} a_{\ell_2,m_2;G}} = (-1)^{m_1+\mu_1}C_{\ell_1;G}C_{\lambda_1;G}\delta_{\ell_1}^{\ell_2}\delta_{\lambda_1}^{\lambda_2}\delta_{m_1}^{-m_2}\delta_{\mu_2}^{-m_2}\\
		&\qquad+(-1)^{-m_1-m_2} C_{\ell_1;G}C_{\ell_2;G} \left(\delta_{\ell_1}^{\lambda_1}\delta_{\ell_2}^{\lambda_2}\delta_{m_1}^{-\mu_1}\delta_{m_2}^{-\mu_2}+\delta_{\ell_1}^{\lambda_2}\delta_{\ell_2}^{\lambda_1}\delta_{m_1}^{-\mu_2}\delta_{m_2}^{-\mu_1}\right),
	\end{align*}
it follows that 
	\begin{align*}
		\mathbb{E}\left[ a_{\ell _{1},m_{1};G}a_{\ell _{2},m_{2};G}a_{\ell _{3},m_{3};2}\right]&  =\frac{1}{\sqrt{C_{\ell_3;G}}} \sum_{\lambda_1} \sum_{\mu_1} \frac{(2\lambda_1+1)}{\sqrt{4\pi(2\ell_3+1)}}\\
		& C_{\lambda_1;G} (-1)^{-m_1-\mu_1}
		\cgbis{\ell_3}{0}{\lambda_1}{0}{\lambda_1}{0}\cgbis{\ell_3}{m_3}{\lambda_1}{\mu_1}{\lambda_1}{-\mu_1}		\\
		&+2\sqrt{\frac{C_{\ell_1;G}C_{\ell_2;G}}{C_{\ell_3;G}}}  \sqrt{\frac{(2\ell_1+1)(2\ell_2+1)}{4\pi(2\ell_3+1)}}   \\
		& (-1)^{-m_1-m_2}
		\cgbis{\ell_3}{0}{\ell_1}{0}{\ell_2}{0}\cgbis{\ell_3}{m_3}{\ell_1}{-m_1}{\ell_2}{-m_2}.
	\end{align*}
	Consider the first addend. Since the Clebsch–Gordan coefficient vanishes unless $m_3=\mu_1-\mu_1=0$, this contribution is proportional to$\delta_{m_3}^0$. Hence, using Equations \eqref{eq:prop1} and \eqref{eq:prop2}, we have that 
	\begin{equation*}
		\begin{split}
			&	\sum_{\mu_1} (-1)^{\lambda_1-\mu_1} \cgbis{\ell_3}{0}{\lambda_1}{\mu_1}{\lambda_1}{-\mu_1}	= \sqrt{2\lambda_1+1} \delta_{\ell_3}^{0}\\
			& \cgbis{0}{0}{\lambda_1}{0}{\lambda_1}{0}=	\frac{(-1)^{\lambda_1} }{\sqrt{2\lambda_1+1} },
		\end{split}
	\end{equation*}
	so that 
	\begin{equation*}
		\begin{split}
			&\frac{1}{\sqrt{C_{\ell_3;G}}} \sum_{\lambda_1} \sum_{\mu_1} \frac{(2\lambda_1+1)}{\sqrt{4\pi(2\ell_3+1)}} C_{\lambda_1;G} (-1)^{-m_1-\mu_1}
			\cgbis{\ell_3}{0}{\lambda_1}{0}{\lambda_1}{0}\cgbis{\ell_3}{m_3}{\lambda_1}{\mu_1}{\lambda_1}{-\mu_1}	\\
			&=\frac{1}{\sqrt{C_{0;G}}} \delta_{\ell_3}^{0}\delta_{m_3}^{0}\sum_{\lambda_1} \frac{2\lambda_1+1}{\sqrt{4\pi}} C_{\lambda_1;G} = 0,
		\end{split}
	\end{equation*}
Since throughout the proposition we consider
\(\ell_3\ge1\), the factor
\(\delta_{\ell_3}^0\)
forces this contribution to vanish.	Using Equation \eqref{eq:wignercj}, the second addend becomes 
	\begin{equation*}
		\begin{split}
			2\sqrt{\frac{ C_{\ell_1;G}C_{\ell_2;G}}{C_{\ell_3;G}}}  &\sqrt{\frac{(2\ell_1+1)(2\ell_2+1)}{4\pi(2\ell_3+1)}  }  (-1)^{-m_1-m_2}
			\cgbis{\ell_3}{0}{\ell_1}{0}{\ell_2}{0}\cgbis{\ell_3}{m_3}{\ell_1}{-m_1}{\ell_2}{-m_2}\\ &= \gamma_{\ell_{1},\ell_{2}}\hal \wigner{\ell_1}{\ell_2}{\ell_3}{m_1}{m_2}{m_3}\wigner{\ell_1}{\ell_2}{\ell_3}{0}{0}{0}.
		\end{split}
	\end{equation*}	
	Finally, we assume $\ell_i \geq 1$ , $i=1,2,3$. The computation for the remaining two first-order terms is identical after permuting the indices (see also \cite{m}).

	It follows that
	\begin{align*}
		&\Ex {a_{\ell _{1},m_{1};G} a_{\ell _{2},m_{2};G}a_{\ell
				_{3},m_{3};2}}+\Ex {a_{\ell _{1},m_{1};G} a_{\ell _{2},m_{2};2}a_{\ell
				_{3},m_{3};G}}+\Ex {a_{\ell _{1},m_{1};2} a_{\ell _{2},m_{2};G}a_{\ell
				_{3},m_{3};G}}\\
		& \qquad \qquad \qquad\qquad \qquad\qquad = 2 \hal\Gammal \begin{pmatrix}
			\ell _{1} & \ell _{2} & \ell _{3} \\ 
			m_{1} & m_{2} & m_{3}%
		\end{pmatrix}%
		\begin{pmatrix}
			\ell _{1} & \ell _{2} & \ell _{3} \\ 
			0 & 0 & 0%
		\end{pmatrix}.
	\end{align*}
Collecting the three first-order contributions yields
\[
\mathbb E[\widetilde a_{\ell_1m_1}
\widetilde a_{\ell_2m_2}
\widetilde a_{\ell_3m_3}]
=
f_{\rm NL}\,\eta_{\ell_1\ell_2\ell_3}
+o_\ell(f_{\rm NL}).
\]
\end{proof}

\begin{proof}[Proof of Proposition \ref{prop:Best}]
	We first compute the expectation of the estimator \eqref{eq:Bispest} to prove \eqref{eq:exb}. Indeed, using \eqref{eq:expwalm} yields 
	\begin{equation*}
		\begin{split}
			\Ex{ \Bispest} & =\sum_{m_{1}=-\el{1}}^{\el{1}}\sum_{m_{2}=-\el{2}}^{\el{2}}\sum_{m_{3}=-\el{3}}^{\el{3}} \wigner{\ell _{1}}{\ell _{2}} {\ell _{3}}{m_{1}} {m_{2}} {m_{3}}  ^2
			\Bisp = \Bisp,
		\end{split}
	\end{equation*}
	where the last equality follows Equation \eqref{eq:prop3}.  

	In order to prove \eqref{eq:varb}, first observe that $\Ex{\Bispest}^2= \fNL^2$. Then, note that  
	\begin{align*}
		\Ex{\widetilde{a}_{\ell _{1},m_{1}}\widetilde{a}_{\ell _{2},m_{2}}%
			\widetilde{a}_{\ell _{3},m_{3}}\cc{\widetilde{a}}_{\ell
				_{1},m_{1}^{\prime }}\cc{\widetilde{a}}_{\ell _{2},m_{2}^{\prime }}%
			\cc{\widetilde{a}}_{\ell _{3},m_{3}^{\prime }}} &\\ 
		=\delta _{m_{1}}^{m_{1}^{\prime }}\delta _{m_{2}}^{m_{2}^{\prime }}\delta
		_{m_{3}}^{m_{3}^{\prime }}  \Ex{ \left\vert \widetilde{a}_{\ell
				_{1},m_{1}}\right\vert ^{2}} &  \Ex{ \left\vert \widetilde{a}%
			_{\ell _{2},m_{2}}\right\vert ^{2}} \Ex{ \left\vert 
			\widetilde{a}_{\ell _{3},m_{3}}\right\vert ^{2}}\\
		=  \delta _{m_{1}}^{m_{1}^{\prime }}\delta _{m_{2}}^{m_{2}^{\prime }}\delta
		_{m_{3}}^{m_{3}^{\prime }}+O\left( \fNL^{2}\right).
	\end{align*}%
	Under the constraint $\ell_1<\ell_2<\ell_3$, every Wick contraction pairing different multipoles vanishes because of the Kronecker deltas in the covariance. 
	Therefore, Wick’s theorem gives	
	\begin{align*}
		\Ex{\Bispest ^2}=&  \sum_{\mm{1}, \mm{2},\mm{3}}\sum_{\mm{1}^{\prime },\mm{2}^{\prime
			},\mm{3}^{\prime}} \wigner{\el{1}} {\el{2}} {\el{3}} {\mm{1}} {\mm{2}} {\mm{3}}\wigner{\el{1}} {\el{2}} {\el{3}} {\mm{1}^\prime} {\mm{2}^\prime} {\mm{3}^\prime}\\
		&\Ex{ \widetilde{a}_{\ell _{1},m_{1}}\widetilde{a}_{\ell
				_{2},m_{2}}\widetilde{a}_{\ell _{3},m_{3}}\cc{\widetilde{a}}_{\ell
				_{1},m_{1}^{\prime }}\cc{\widetilde{a}}_{\ell _{2},m_{2}^{\prime }}
			\cc{\widetilde{a}}_{\ell _{3},m_{3}^{\prime }}}\\
		=&  \sum_{\mm{1}, \mm{2},\mm{3}}\wigner{\el{1}} {\el{2}} {\el{3}} {\mm{1}} {\mm{2}} {\mm{3}}^2 + O\left( \fNL^2\right)\\
		=& 1 +  O\left( \fNL^2\right).
	\end{align*}
As far as Equation \eqref{eq:cqb} is concerned, we follow again the arguments of \cite{m}. Using the definition of fourth cumulant yields 
\begin{equation*}
	\begin{split}
		&\cumq{\Bispest} \\ & = \sum_{m_1,m_2,m_3}\sum_{m_4,m_5,m_6}\sum_{m_7,m_8,m_9}\sum_{m_{10},m_{11},m_{12}}\wigner{\el{1}} {\el{2}} {\el{3}} {\mm{1}} {\mm{2}} {\mm{3}}\\ & \quad \wigner{\el{1}} {\el{2}} {\el{3}} {\mm{4}} {\mm{5}} {\mm{6}}\wigner{\el{1}} {\el{2}} {\el{3}} {\mm{7}} {\mm{8}} {\mm{9}}\wigner{\el{1}} {\el{2}} {\el{3}} {\mm{10}} {\mm{11}} {\mm{12}}\\ &\cumq{\tilde a_{\el{1},\mm{1}}\tilde a_{\el{2},\mm{2}}\tilde a_{\el{3},\mm{3}} \tilde a_{\el{1},\mm{4}}\tilde a_{\el{2},\mm{5}}\tilde a_{\el{3},\mm{6}} \tilde a_{\el{1},\mm{7}}\tilde a_{\el{2},\mm{8}}\tilde a_{\el{3},\mm{9}}\tilde a_{\el{1},\mm{10}}\tilde a_{\el{2},\mm{11}} \tilde a_{\el{3},\mm{12}} }.
	\end{split}
\end{equation*}
By the diagram formula for fourth-order cumulants (see \cite{m,MaPeCUP}), only four connected pairing configurations contribute:
\begin{equation*}
	\begin{split}
		\cumq{\Bispest}  & = 6 \sum_{m_1,m_2,m_3}\sum_{m_4,m_5,m_6}\wigner{\el{1}} {\el{2}} {\el{3}} {\mm{1}} {\mm{2}} {\mm{3}} \wigner{\el{1}} {\el{2}} {\el{3}} {\mm{1}} {\mm{4}} {\mm{5}}\\
		&	\quad \wigner{\el{1}} {\el{2}} {\el{3}} {\mm{6}} {\mm{2}} {\mm{5}}\wigner{\el{1}} {\el{2}} {\el{3}} {\mm{6}} {\mm{4}} {\mm{3}}\\
		&+6 \sum_{m_1,m_2,m_3}\sum_{m_4,m_5,m_6}\wigner{\el{1}} {\el{2}} {\el{3}} {\mm{1}} {\mm{2}} {\mm{3}} \wigner{\el{1}} {\el{2}} {\el{3}} {\mm{1}} {\mm{4}} {\mm{3}}\\
		&	\quad \wigner{\el{1}} {\el{2}} {\el{3}} {\mm{5}} {\mm{4}} {\mm{6}}\wigner{\el{1}} {\el{2}} {\el{3}} {\mm{5}} {\mm{2}} {\mm{6}}\\
		&+6 \sum_{m_1,m_2,m_3}\sum_{m_4,m_5,m_6}\wigner{\el{1}} {\el{2}} {\el{3}} {\mm{1}} {\mm{2}} {\mm{3}} \wigner{\el{1}} {\el{2}} {\el{3}} {\mm{1}} {\mm{2}} {\mm{4}}\\
		&	\quad \wigner{\el{1}} {\el{2}} {\el{3}} {\mm{5}} {\mm{6}} {\mm{4}}\wigner{\el{1}} {\el{2}} {\el{3}} {\mm{5}} {\mm{6}} {\mm{3}}\\
		&+6 \sum_{m_1,m_2,m_3}\sum_{m_4,m_5,m_6}\wigner{\el{1}} {\el{2}} {\el{3}} {\mm{1}} {\mm{2}} {\mm{3}} \wigner{\el{1}} {\el{2}} {\el{3}} {\mm{4}} {\mm{2}} {\mm{3}}\\
		&	\quad \wigner{\el{1}} {\el{2}} {\el{3}} {\mm{4}} {\mm{5}} {\mm{6}}\wigner{\el{1}} {\el{2}} {\el{3}} {\mm{1}} {\mm{5}} {\mm{6}}. 
	\end{split}
\end{equation*}
The first sum is equal to the Wigner $6j$-symbol (see \cite[p. 417]{vmk}). The remaining three sums are evaluated using the orthogonality identity \eqref{eq:wignerup}, yielding $(2\ell_i+1)^{-1}$, $i=1,2,3$; see \cite[Theorem 9.7]{MaPeCUP}. This completes the proof of \eqref{eq:cqb}.
\end{proof}

\subsection{Main results}
This section gathers the proofs of our main results, along with the key auxiliary lemmas and technical estimates that are essential for establishing them.
\begin{proof}[Proof of Lemma \ref{lemma:integral}]
	We begin by applying the following change of variables to the integrand in \eqref{eq:genint}:
	\begin{equation*}
		x_1=ut; \quad x_2=v t, \quad x_3=t,
	\end{equation*}
	where $t$ represents the overall scale and $(u,v)$ describe the normalized shape of the triangle. Since $\Delta$ is homogeneous of degree $4$, we have
	\begin{equation*}
		\Delta(x_1,x_2,x_3)= t^4	\Delta(u,v,1) = t^4 Q(u,v).
	\end{equation*}
	This factorization is helpful in integration, as it separates scale from shape. \\
	Under this change of variables, the integrand must be multiplied by the Jacobian $J$ determinant associated with the transformation, which accounts for the change in volume element. The Jacobian $J$ is given by
	\begin{equation*}
		J=\begin{pmatrix}
			t & 0 & u\\
			0 & t & v\\
			0 & 0 & 1 
		\end{pmatrix},
	\end{equation*}
	and $\text{det}J = t^2$.\\
With regard to the integration domain, the scaling variable satisfies
\[
t\in\left(\frac{r}{u},1\right).
\] 
Indeed, the constraints $x_3=t<1$ and $x_1=ut>r$ imply $t>r/u$. We now determine the corresponding domain for the variables $(u,v)$. Since $x_1,x_2<x_3$, it follows that
\[
0<u<1,\qquad 0<v<1.
\]
Moreover, from $x_1=ut>r$ and $t<1$, we deduce that $u>r$. Finally, because $x_2>x_1$, we have
\[
x_2=vt>x_1=ut,
\]
and hence $v>u$.
The triangle inequalities become
\[
|u-v|<1<u+v.
\]
Since $v>u$, the first inequality reduces to $v-u<1$, which is automatically satisfied because $0<u<v<1$. The second inequality is equivalent to
\[
v>1-u.
\]
Combining these conditions yields
\[
\max(u,1-u)<v<1.
\]
Since $\max(u,1-u)=1-u$ for $u<1/2$ and $\max(u,1-u)=u$ for $u>1/2$, the integration domain is precisely $\Omega_r=\Omega_{1,r}\cup\Omega_{2,r}$, where $\Omega_{1,r}$ and $\Omega_{2,r}$ are defined in \eqref{eq:omegadomain}.
Hence, by defining $\lambda=a+b+c-4d+3$, Equation \eqref{eq:genint} becomes
	\begin{equation*}
		\begin{split}
			D_r(a,b,c,d) & = \int_{\Omega_r} \frac{u^av^b}{Q^{d}(u,v)} \left(\int_{\frac{r}{u}}^{1} t^{\lambda-1} \diff t \right) \diff u \diff v
		\end{split}.
	\end{equation*}
Evaluating the inner integral with respect to the scaling variable t yields \eqref{eq:Gfinale}. The resulting integral depends only on the shape variables $(u,v)$ and is finite over $\Omega_r$.
\end{proof}

\begin{proof}[Proof of Lemma \ref{LemmaFlashGordon}]
	We use the explicit formula for the Wigner $3j$-symbol with $m_1=m_2=m_3=0$ (see \cite[Eq. 3.61]{MaPeCUP}):
	\begin{equation}\label{eq:wigner1}
		\begin{split}
			\wigner{\ell _{1}}{\ell _{2}}{\ell _{3}}{0}{0}{0}
			&=\left( -1\right) ^{\frac{\ell _{1}+\ell _{2}+\ell _{3}}{2}}\frac{\left( \frac{\ell
					_{1}+\ell _{2}+\ell _{3}}{2}\right) !}{\left( \frac{\ell _{1}+\ell _{2}-\ell
					_{3}}{2}\right) !\left( \frac{\ell _{1}-\ell _{2}+\ell _{3}}{2}\right)
				!\left( \frac{-\ell _{1}+\ell _{2}+\ell _{3}}{2}\right) !} \\
			&\times \sqrt{\frac{\left( \ell _{1}+\ell _{2}-\ell _{3}\right) !\left(
					\ell _{1}-\ell _{2}+\ell _{3}\right) !\left( -\ell _{1}+\ell _{2}+\ell
					_{3}\right) !}{\left( \ell _{1}+\ell _{2}+\ell _{3}+1\right) !}}\text{ ;}
		\end{split}
	\end{equation}%
	Applying Stirling’s approximation to each factorial in \eqref{eq:wigner1} gives 
	\begin{equation*}\label{eq:wigner2}
		\begin{split}
			& \wigner{\ell _{1}}{\ell _{2}}{\ell _{3}}{0}{0}{0}
			\simeq \left( -1\right) ^{\frac{\ell _{1}+\ell _{2}+\ell _{3}}{2}}\sqrt{\frac{2e}{\pi}} \left(\frac{\ell_1+\ell_2+\ell_3}{\ell_1+\ell_2+\ell_3+1}\right)^{\frac{\ell_1+\ell_2+\ell_3-1}{2}}\\
			& \qquad \times \left( \left( \ell _{1}+\ell
			_{2}+\ell _{3}+1\right) \left( \ell _{1}+\ell _{2}-\ell _{3}\right) \left(
			\ell _{1}-\ell _{2}+\ell _{3}\right) \left( -\ell _{1}+\ell _{2}+\ell
			_{3}\right) \right) ^{-\frac{1}{4}}  .
		\end{split}
	\end{equation*}
	For large $\ell_1,\ell_2,\ell_3$, the last factor in the previous formula is asymptotically equivalent to  $\Delta\left( \ell _{1},\ell _{2},\ell _{3}\right) ^{-\frac{1}{4}}$.
	Let us now define $\rho=\ell_1+\ell_2+\ell_3$, so that 
	\begin{equation*}
		\begin{split}
			\left(\frac{\ell_1+\ell_2+\ell_3}{\ell_1+\ell_2+\ell_3+1}\right)^{\frac{\ell_1+\ell_2+\ell_3-1}{2}} &= 	\left(\frac{\rho}{\rho+1}\right)^{\frac{\rho-1}{2}}\\
			&= 	\left(1-\frac{1}{\rho+1}\right)^{\frac{\rho-1}{2}}
		\end{split} .
	\end{equation*}
	Note that 
	\begin{equation*}
		\begin{split}
			\log \left(1-\frac{1}{\rho+1}\right)^{\frac{\rho-1}{2}} & = \frac{\rho-1}{2}	\log \left(1-\frac{1}{\rho+1}\right)\\ 
			&\underset{\rho \rightarrow \infty}
			{\approx	}\frac{\rho-1}{2}\frac{-1}{\rho+1} \\
			& = -\frac{1}{2}+O\left(\frac{1}{\rho}\right).	
		\end{split}
	\end{equation*}
	Thus 
	\begin{equation*}
		\left(\frac{\ell_1+\ell_2+\ell_3}{\ell_1+\ell_2+\ell_3+1}\right)^{\frac{\ell_1+\ell_2+\ell_3-1}{2}} \approx e^{-\frac{1}{2}},
	\end{equation*}
	and we obtain therefore \eqref{FlashGordon1}. 
\end{proof}

\begin{proof}[Proof of Lemma \ref{lemma:eta}]
	under Condition \ref{eq:powercond},as $L \rightarrow \infty$, we have the asymptotic approximation
	\begin{align*}
		\etal^2 = \frac{(2\ell_1+1)(2\ell_2+1)(2\ell_3+1)}{(4\pi)^{\frac{3}{2}}}
		\begin{pmatrix}
			\ell _{1} & \ell _{2} & \ell _{3} \\ 
			0 & 0 & 0%
		\end{pmatrix}%
		^{2}\Gamma _{\ell _{1}\ell _{2}\ell _{3}}^{2} 
	\end{align*}
	with 
	\begin{align*}
		\etal^{2,\text{lim}} =C_{\eta^2} \frac{\ell _{1}^{1-\alpha}\ell_{2}^{1-\alpha}\ell _{3}^{1+\alpha}}  {\Delta\left(\ell_1,\ell_2, \ell_3\right)^{\frac{1}{2}}} .
	\end{align*}%
	Indeed, for $\ell_1 <\ell_2<\ell_3$, under Condition~\ref{eq:powercond}, $\Gamma _{\ell _{1}\ell_{2}\ell _{3}}^{2}$ is dominated by the term $C_{\ell_1}C_{\ell_2}/C_{\ell_3}$, which satisfies the asymptotic relation 
	\begin{equation*}
		A \left(\frac{\ell_1\ell_2}{\ell_3}\right)^{-\alpha}. 
	\end{equation*}
	Moreover, by Lemma \ref{LemmaFlashGordon}, the squared Wigner $3j$-symbol $\begin{pmatrix}
		\ell _{1} & \ell _{2} & \ell _{3} \\ 
		0 & 0 & 0%
	\end{pmatrix}%
	^{2}$ is asymptotically equivalent to 
	\begin{equation*}
		\frac{2}{\pi} \Delta\left(\ell_1,\ell_2, \ell_3\right)^{-\frac{1}{2}}.
	\end{equation*}
	Finally, $\hal^2$ is approximated by 
	\begin{equation*}
		4\left(\frac{2}{ \pi}\right)^3\ell_1\ell_2\ell_3.
	\end{equation*}
	Therefore,
	\begin{align*}
		\Seta^{\text{lim}}= C_{\eta^2} \summinus \etal^{2,\text{lim}}.
	\end{align*}
An analogous argument yields
	\begin{equation*}
		\kappal =  \frac{(2\ell_1+1)(2\ell_2+1)^2(2\ell_3+1)^2}{(4\pi)^{3}}
		\begin{pmatrix}
			\ell _{1} & \ell _{2} & \ell _{3} \\ 
			0 & 0 & 0%
		\end{pmatrix}%
		^{4}\Gamma _{\ell _{1}\ell _{2}\ell _{3}}^{4}, 
	\end{equation*}
	with
	\begin{align*}
		\kappal^{\text{lim}} =C_{\kappa} \frac{\ell _{1}^{1-2\alpha}\ell_{2}^{2\left(1-\alpha\right)}\ell _{3}^{2\left(1+\alpha\right)}}  {\Delta\left(\ell_1,\ell_2, \ell_3\right)} ,
	\end{align*}
	such that  $\Skappa$ can be approximated with 
	\begin{align*}
		\Skappa^{\text{lim}}= C_{\kappa} \summinus \kappal^{\text{lim}}.
	\end{align*}
	In both the cases, as $L$ grows to infinity, since the summands are evaluations of smooth functions on the lattice $\Lambda_L$, the sums are Riemann sums associated with the corresponding integrals. Hence, $\Seta^{\text{lim}}$ and $\Skappa^{\text{lim}}$ are well-approximated by the Riemann integrals
	\begin{align*}
		& \int_{A_r}   \frac{ \left(Lx_1 \right)^{1-\alpha}  \left(Lx_2 \right)^{1-\alpha} \left(Lx_3 \right)^{1+\alpha}} {\Delta\left(Lx_1,Lx_2, Lx_3\right)^{\frac{1}{2}}} L^3 \diff x_1 \diff x_2 \diff x_3 = 	L^{4-\alpha}I_{\eta^2;r}\left(\alpha\right),\\
		& \int_{A_r}   \frac{\left(Lx_1 \right)^{1-2\alpha}\left(Lx_2 \right)^{2\left(1-\alpha\right)}\left(Lx_3 \right)^{2\left(1+\alpha\right)}}  {\Delta\left(Lx_1,Lx_2, Lx_3\right)} L^3 \diff x_1 \diff x_2 \diff x_3 = 	L^{4-2\alpha}I_{\kappa,r}\left(\alpha\right)
	\end{align*}
	where $A_r$ is given by \eqref{eq:intdomain}, since
\[
\Delta(Lx_1,Lx_2,Lx_3)
=
L^4\Delta(x_1,x_2,x_3),
\]
the integrands scale as
\[
L^{1-\alpha+1-\alpha+1+\alpha-2+3}
=
L^{4-\alpha},
\]
and similarly for the second integral. 
	Applying Lemma \ref{lemma:integral} with $a=b=1-\alpha$, $c=1+\alpha$ and $d=1/2$,  we obtain $\lambda = 4-\alpha$ and 
	\begin{equation*}
		I_{\eta^2,r} (\alpha) = \frac{1}{4-\alpha} \int_{\Omega_r} \frac{(uv)^{1-\alpha}}{Q(u,v)^{\frac{1}{2}}} \left(1-\left(\frac{r}{u}\right)^{4-\alpha} \right) \diff u \diff v.
	\end{equation*}
	Analogously, for $a=1-2\alpha$, $b=2(1-\alpha)$, $c=2(1+\alpha)$ and $d=1$,  we obtain $\lambda = 4-2\alpha$ and 
	\begin{equation*}
		I_{\kappa,r} (\alpha) = \frac{1}{2(2-\alpha)} \int_{\Omega_r} \frac{u^{1-2\alpha}v^{2(1-\alpha)}}{Q(u,v)} \left(1-\left(\frac{r}{u}\right)^{2(2-\alpha)} \right) \diff u \diff v.
	\end{equation*}
\end{proof}

\begin{proof}[Proof of Proposition \ref{lemmaest}]
	Combining Proposition \ref{prop:Best} with \eqref{eq:Bisp} and \eqref{eq:vec} yields
	\begin{equation*}
		\begin{split}
			\Ex{\fNLest} & =  \Ex{\frac{H^T Y}{H^T H}}\\
			&=\frac{H^T \Ex{Y}}{H^T H}\\
			&=\fNL \frac{H^T H}{H^T H} + o\left(\fNL\right)\\
			&=\fNL + o\left(\fNL\right),
		\end{split}
	\end{equation*}
	Proposition~\ref{prop:Best} implies that each component of Y has asymptotic variance one and distinct components are asymptotically uncorrelated. Hence,
	\[
	\Var{Y}=\bm{1}_{N_L},
	\]
where $\bm{1}_{N_L}$ is the identity matrix of dimension $N_L$. Hence, we have that	
	\begin{equation*}
		\begin{split}
			\Var{ \fNLest } &= \Var{\frac{H^TY}{H^TH}}\\
			&=  \frac{1}{\left(H^TH\right)^2} H^T \Var{Y} H \\ 
			&=  \frac{1}{\left(H^TH\right)^2} H^T\bm{1}_{N_L} H \\
			& =  \frac{1}{H^TH}\\
			& =  \frac{1}{\summinus \etal ^2}\\
			&= \Seta^{-1}.
		\end{split}
	\end{equation*}
	Using Lemma \ref{lemma:eta} yields
	\begin{equation*}
		\begin{split}
			\Var{ \fNLest } &= \frac{1}{C_{\eta^2}I_{\eta^2;r}\left(\alpha\right)} L^{\alpha-4}+ o\left(L^{\alpha-4}\right)\\
			& = \sigma_{\fNL}^2 L^{\alpha-4}+ o\left(L^{\alpha-4}\right).
		\end{split}
	\end{equation*}
Then, since the components of $Y$ are uncorrelated, all mixed fourth-order cumulants vanish, and multilinearity yields
	\begin{equation*}
		\begin{split}
			\cumq{\fNLest} & = \cumq{\frac{H^TY}{H^TH}}\\
			&= \frac{1}{\left(H^TH\right)^4}	\cumq{H^TY} \\
			& =  \frac{1}{\left(\summinus \etal ^2\right)^4}  \summinus \etal ^4 \cumq{\Bispest}\\
			& \leq  \frac{1}{\left(\summinus \etal ^2\right)^4}  \summinus \kappa_{\ell_1,\ell_2,\ell_3} 
			\frac{12}{2\ell_1+1}\\
			&= \frac{\Skappa}{\Seta^4}
		\end{split} 
	\end{equation*}
	Applying Lemma \ref{lemma:eta} completes the proof.
\end{proof}

\begin{proof}[Proof of Theorem \ref{thm:main}]
	Since the estimator $\fNLest$ is an ordinary least squares estimator, its fluctuations are entirely determined by the stochastic component $E$, which is a vector of uncorrelated random variables with asymptotic variance one.
	By the fourth-moment bound of \cite{noupebook} (see also \cite{MaPeCUP}), the following bound holds
	\begin{equation*}
		\dW\left( \frac{\fNLest - \fNL}{\sqrt{\Var{\fNLest}}} , Z \right) \leq \frac{2\sqrt{2}}{3}
		\sqrt{\frac{\cumq{\fNLest} }{\left(\Var{\fNLest}\right)^2 }}.
	\end{equation*}
	Using Proposition \ref{lemmaest}, we have that 
	\begin{equation*}
		\begin{split}
			\frac{\cumq{\fNLest} }{\left(\Var{\fNLest}\right)^2} &= \frac{K_{\fNL}}{\sigma^4_{\fNL}} L^{-4}\\
			&= \frac{C_{\kappa}I_{\kappa,r}(\alpha)}{C_{\eta^2}^2I^{2}_{\eta^2,r}(\alpha)} L^{-4}.
		\end{split}
	\end{equation*}
Moreover,
	\begin{equation*}
		\begin{split}
			\frac{C_{\kappa}}{C_{\eta^2}^2} = \frac{3A^2}{8\pi^8}\frac{16\pi^8}{A^2}=6.
		\end{split}
	\end{equation*}
	Thus, 
	\begin{align*}
		\frac{2\sqrt{2}}{3}\sqrt{\frac{\cumq{\fNLest} }{\left(\Var{\fNLest}\right)^2 }} & =\frac{4}{\sqrt{3}}\frac{I^{\frac{1}{2}}_{\kappa;r}\left(\alpha\right)}{I_{\eta^2;r}\left(\alpha\right)} L^{-2}.
	\end{align*}
\end{proof}

\section*{Funding}
This work was partially supported by the PRIN 2022 project GRAFIA (Geometry of Random Fields and its Applications), funded by the Italian Ministry of University and Research (MUR).

\section*{Acknowledgments}
The author gratefully acknowledges D. Marinucci for his insightful ideas, valuable suggestions, and stimulating discussions, which were pivotal to the completion of this work.

\section*{Declaration of generative AI use}
The author used OpenAI's ChatGPT as a language-editing and writing assistance tool during the preparation of this manuscript. The tool was employed to improve the clarity, grammar, and organization of the text, and to assist with LaTeX formatting and presentation. All mathematical results, proofs, numerical methods, simulations, scientific interpretations, and conclusions were developed, verified, and approved by the author, who takes full responsibility for the content of the manuscript.

\medskip
\bibliographystyle{abbrv}
\bibliography{bibliografia}

@article{bm,
author={Baldi, P. and Marinucci, D.}, 
title={Some Characterizations of the spherical harmonics coefficients for isotropic random fields},
journal={Statist. Probab. Lett.},
volume={77},
pages={490--496},
year={2007},
}

@article{bartolo,
	author={Bartolo, N. and Komatsu, E. and Matarrese, S. and Riotto, A.},
	year={2004},
	title={Non-{G}aussianity from inflation: theory and observations}, 
	journal={Physical Reports},
	volume={402}, 
	pages={103--266},
}

@book{dode2004, 
author={Dodelson, S.},
year={2003}, 
title={Modern
cosmology},
publisher={Academic Press},
}

@article{dlm, 
author={Durastanti, C. and Lan, X. and Marinucci, D.},
year={2014},
title={Gaussian semiparametric estimates on the unit sphere}, 
journal={Bernoulli}, 
volume={20}, 
issue={1}, 
pages={28--77},
}

@article{dp19,
	author={Durastanti, Cl and Patschkowski, T.},
	title={Aliasing effects for random fields over spheres of arbitrary dimension},	
	journal={Electron. J. Statist.}, 
	volume={13}, 
	issue={2},
	year={2019}, 
	pages={3297--3335},
}

@article{hu,
	author={Hu, W.},
	year= {2001},
	title={The angular trispectrum of the {CMB}},
	journal={Phys. Rev. D},
	volume={64}, 
	pages={083005},
}

@article{ksw05,
author={Komatsu, E. and Spergel, D. N. and Wandelt, B. D.},
year={2005}, 
title={Measuring {P}rimordial {N}on-{G}aussianity in the {C}osmic {M}icrowave {B}ackground}, 
journal={ApJ}, 
volume={634}, 
issue={14}, 
}

@article{ls15,
 author={Lang, A. and Schwab, C.},
 year={2015},
 title={Isotropic Gaussian random fields on the sphere: Regularity, fast simulation and stochastic partial differential equations},
 journal={Ann. Appl. Prob.},
 volume={25},
 issue={6},
 pages={3047--3094},
}

@article{LSFS10,
	author = {Liguori, M. and Sefusatti, E. and Fergusson, J. R. and Shellard, E. P. S.},
	title = "{Primordial non-Gaussianity and Bispectrum Measurements in the Cosmic Microwave Background and Large-Scale Structure}",
	journal = "Adv. Astron.",
	pages = "980523",
	year = "2010"
}

@article{m,
author={Marinucci, D.}, 
year={2006}, 
title={High-resolution asymptotics for the angular bispectrum of spherical random fields}, 
journal={Ann. Statist.}, 
volume={34}, 
pages={1--41}, 
}

@article{M2008,
author={Marinucci, D.},
year={2008}, 
title={A central limit theorem and higher order results for the angular bispectrum},
journal={Probab. Theory Related Fields}, 
volume={3-4}, 
pages={389--409}, 
}

@book{MaPeCUP,
author={Marinucci, D. and Peccati, G.},
year={2011}, 
title={Random Fields on the Sphere: Representations, Limit Theorems and Cosmological Applications}, 
publisher={Cambridge University Press},
}

@article{marpec2,
author={Marinucci, D. and Peccati, G.},
year={2010},
title={Ergodicity and {G}aussianity for spherical random fields}, 
journal={J. Math. Phys.}, 
volume={51}, 
pages={043301}, 
}

@article{marpec4,
author={Marinucci, D. and Peccati, G.},
year={2010},
title={Group representations and high-resolution central limit theorems for subordinated spherical random fields},
journal={Bernoulli}, 
volume={16}, 
issue={3}, 
pages={798--824}, 
}

@article{baas,
	author = {Meerburg, P. D. and Green, D. and Flauger, R. and Wallisch, B. and Marsh, M.C. D. and Pajer, E. and Goon, G. and Dvorkin, C. and Dizgah, A. M. and Baumann, D. and Pimentel, G. L. and Foreman, S. and Silverstein, E. and Chisari, E. and Wandelt, B. and Loverde, M. and Slosar, A.},
	journal = { Bull.Am.Astron.Soc. },
	number = {3},
	year = {2019},
	title = {Primordial {Non}-{Gaussianity}},
	volume = {51},
}

@article{Mun14,
author={M\"unchmeyer, M. and Bouchet F. and Jackson M.G. and Wandelt B.} ,
year={2014}, 
title={The {Komatsu Spergel Wandelt} estimator for oscillations in the cosmic microwave background bispectrum}, 
journal={A\&A}, 
volume={570}, 
issue={A94}, 
pages={1925--1956},
}

@article{nourdinpeccati,
author={Nourdin, I. and Peccati, G.},
year={2009},
title={Stein's method on {W}iener chaos}, 
journal={Probab. Theory Related Fields}, 
volume={145}, 
issue={1-2}, 
pages={75--118}, 
}

@book{noupebook,
author={Nourdin, I. and Peccati, G.} ,
year={2012}, 
title={Normal approximations using Malliavin calculus: from Stein's method to universality}, 
publisher={Cambridge University Press},
}

@article{npopt,
	author = {Nourdin, I. and Peccati, G.},
	journal = {Proc. Am. Math. Soc.},
	number = {7},
	pages = {3123--3133},
	title = {The Optimal Fourth Moment Theorem},
	volume = {143},
	year = {2015}
}

@book{hardy79,
	author={Hardy, G. H. and Wright, E. M.}, 
	year={1979},
	title={An Introduction to the Theory of Numbers (5th ed.)},
	publisher={Oxford} ,
}

@book{steinweiss,
author={Stein, E. M. and Weiss, G.},
year={1971},
title={Introduction to Fourier analysis on Euclidean spaces},
publisher={Princeton University Press},
}

@book{szego,
author={Szego, G.},
year={1975}, 
title={Orthogonal Polynomials},
publisher={Amer. Math. Soc. Colloq. Publ.},
edition={4th}, 
}

@book{vmk,
author={Varshalovich, D. A. and Moskalev, A. N. and Khersonskii V. K.},
year={1988},
title={Quantum theory of angular momentum},
publisher={World Scientific},
}

@book{vilenkin,
author={Vilenkin, N. J. and Klimyk, A. U.},
year={1991}, 
title={Representation of Lie groups and special functions},
publisher={Kluwer},
}

@book{yadrenko,
	author={Yadrenko, M.I.},
	year={1983},
	title={Spectral theory of random fields},
	publisher={Optimization Software Inc.},  	
}

@article{yadavwandelt10,
author = {Yadav, A. P. S. and Wandelt, B. D.},
title = {Primordial Non-{G}aussianity in the {C}osmic {M}icrowave {B}ackground},
journal = {Adv. Astron.},
volume = {2010},
number = {1},
pages = {565248},
year = {2010}
}

@article{komatsu10,
    author = "Komatsu, E.",
    title = "{Hunting for Primordial Non-{G}aussianity in the {C}osmic {M}icrowave {B}ackground}",
    journal = "Class. Quant. Grav.",
    volume = "27",
    pages = "124010",
    year = "2010"
}

@article{cagliari24,
year = {2024},
volume = {2024},
number = {08},
pages = {036},
author = {Cagliari, M. S. and Castorina, E. and Bonici, M. and Bianchi, D.},
title = {Optimal constraints on {P}rimordial non-{G}aussianity with the e{BOSS} {DR16} quasars in {F}ourier space},
journal = {J. Cosmol. Astropart. Phys.},
}

@article{planck2020a,
author = {{Planck Collaboration}},
title = "{Planck 2018 results. IX. Constraints on primordial non-Gaussianity}",
journal = {{A\&A}},
year = 2020,
volume = {641},
pages = {A9},
}

@article{planck2020b,
author = {{Planck Collaboration}},
title = "{Planck 2018 results. VII. Isotropy and statistics of the CMB}",
journal = {{A\&A}},
year = 2020,
volume = {61},
pages = {A7},
}
\end{document}